\documentclass[11pt,a4paper,dvipsnames]{scrartcl}
\usepackage{amsmath,amsthm,amssymb,url}

\usepackage[english]{babel}
\usepackage{geometry}

\usepackage[backend=biber, giveninits=true, style=ieee, dashed=false, sorting = nty,citestyle=numeric, sortcites]{biblatex}

\AtEveryBibitem{\clearfield{doi} 
\clearfield{month}
\clearfield{day}
\clearfield{issn}}

\usepackage{bm}
\usepackage{microtype}
\usepackage{mathtools}
\usepackage{csquotes}

\usepackage{enumitem}
\setlist[enumerate]{itemsep=0mm,parsep=2mm,topsep=5pt,label=\textit{(\alph*)}}
\setlist[itemize]{itemsep=1mm,parsep=2mm,topsep=6pt}

\usepackage{pgf,tikz,pgfplots}
\pgfplotsset{compat=1.17}
\usetikzlibrary{shapes.geometric}
\usetikzlibrary{cd}
\usetikzlibrary{patterns}
\usetikzlibrary{decorations.pathreplacing}
\usepackage{tkz-euclide}

\usepackage{hyperref}
\hypersetup{
linkcolor  = black,
citecolor  = black,
colorlinks = true,
}
\usepackage[capitalise, noabbrev, nameinlink, nosort]{cleveref}
\crefname{equation}{}{}
\crefname{conjecture}{Conjecture}{Conjectures}

\AddToHook{env/lemma/begin}{\crefalias{theorem}{lemma}}
\AddToHook{env/conjecture/begin}{\crefalias{theorem}{conjecture}}
\AddToHook{env/proposition/begin}{\crefalias{theorem}{proposition}}

\theoremstyle{definition}
\newtheorem{theorem}{Theorem}[section]
\newtheorem{proposition}[theorem]{Proposition}
\newtheorem{lemma}[theorem]{Lemma}

\newtheorem{conjecture}[theorem]{Conjecture}

\newcommand{\rat}{{\mathbb{Q}}}

\newcommand{\R}{{\mathbb{R}}}
\newcommand{\Rd}{{\mathcal{R}}_d}

\newcommand{\cR}{{\mathcal{R}}}
\newcommand{\sm}{-}%

\DeclareMathOperator{\rk}{rk}

\newcommand{\sharedstressnullity}[2]{k_{#1}(#2)}

\newcommand{\sharedstressrank}[2]{s_{#1}(#2)}

\begin{document}
\title{Sparsity,
Stress-Independence and  Globally Linked Pairs in Graph Rigidity Theory}

\author{D\'aniel Garamv\"olgyi\thanks{
HUN-REN Alfréd Rényi Institute of Mathematics, Reáltanoda utca 13-15, Budapest, 1053, Hungary, and the HUN-REN-ELTE Egerv\'ary Research Group
on Combinatorial Optimization, P\'azm\'any P\'eter s\'et\'any 1/C, 1117 Budapest, Hungary.
e-mail: \texttt{daniel.garamvolgyi@ttk.elte.hu}}
\and
Bill Jackson\thanks{Queen Mary University of London, London, E1 4NS, UK. e-mail: \texttt{b.jackson@qmul.ac.uk}}
\and
Tibor Jord\'an\thanks{Department of Operations Research, ELTE E\"otv\"os Lor\'and University, and the HUN-REN-ELTE Egerv\'ary Research Group
on Combinatorial Optimization, P\'azm\'any P\'eter s\'et\'any 1/C, 1117 Budapest, Hungary.
e-mail: \texttt{tibor.jordan@ttk.elte.hu}}
}

\date{}
\maketitle

\begin{abstract}
      A graph 
      is $\mathcal{R}_d$-independent (resp.\ $\cR_d$-connected) if its $d$-dimensional generic rigidity matroid is free (resp.\ connected). A result of Maxwell from 1867 implies that every $\mathcal{R}_d$-independent graph  satisfies the sparsity condition $|E(H)|\leq d|V(H)|-\binom{d+1}{2}$ for all subgraphs $H$  with at least $d+1$ vertices. Several other families of graphs $G$ arising naturally in rigidity theory, such as minimally globally $d$-rigid graphs,  are known to satisfy the bound $|E(G)|\leq (d+1)|V(G)|-\binom{d+2}{2}$. We unify and extend these results by considering
      the family of $d$-stress-independent graphs which 
      includes many of these families. We show that every $d$-stress-independent graph is $\mathcal{R}_{d+1}$-independent.  A key ingredient in our proofs is the concept of $d$-stress-linked pairs of vertices. 
      We derive a new sufficient condition for $d$-stress linkedness and use it to obtain a similar condition for a pair of vertices of a graph to be globally $d$-linked. This result 
      strengthens a result of Tanigawa on globally $d$-rigid graphs. We also show that every minimally $\cR_d$-connected  graph $G$ is $\mathcal{R}_{d+1}$-independent and that the only subgraphs of $G$ that can satisfy Maxwell's criterion for $\mathcal{R}_{d+1}$-independence with equality are copies of $K_{d+2}$. Our results give affirmative answers to two conjectures
      in graph rigidity theory.
\end{abstract}

\section{Introduction}

In this paper we investigate the relationship between rigidity and global rigidity of graphs, as well as unify and extend known sparsity results for various graph classes appearing in combinatorial rigidity theory. A graph is \emph{$d$-rigid} (resp.\ \emph{globally $d$-rigid}) if its generic realisations in $\R^d$ cannot be deformed continuously (resp.\ arbitrarily) such that the edge lengths are preserved. (Precise definitions are given in the next section.)   
Rigidity and global rigidity are both well-studied graph properties with a number of applications, both within and outside mathematics. From a combinatorial perspective, $d$-rigidity can be seen as a matroidal analogue of $d$-connectivity. Indeed, it is well-known that every $d$-rigid graph on at least $d+1$ vertices is $d$-connected, and that the edge sets of $d$-rigid graphs on $n$ vertices are the spanning sets of the \emph{$d$-dimensional generic rigidity matroid} $\mathcal{R}_d(K_n)$ of the complete graph $K_n$.

In contrast, there is no matroid underlying global rigidity, which makes the combinatorial study of globally $d$-rigid graphs more challenging. To circumvent this difficulty, one can try to find necessary and sufficient conditions for global $d$-rigidity in terms of $d$-rigidity.
It is immediate that every globally $d$-rigid graph is $d$-rigid, and by a seminal result of Hendrickson~\cite{hendrickson_1992}, globally $d$-rigid graphs on at least $d+2$ vertices are also \emph{redundantly $d$-rigid}, which means that they remain rigid after the removal of any edge. In the other direction, we have the following key result of Tanigawa~\cite{tanigawa_2015}. We say that a $d$-rigid graph $G$ is \emph{vertex-redundantly $d$-rigid} if $G-v$ is $d$-rigid for every vertex $v$ of $G$. 

\begin{theorem}\label{theorem:shinichi}\cite{tanigawa_2015}
      Every vertex-redundantly $d$-rigid graph is globally $d$-rigid. 
\end{theorem}

We can use \cref{theorem:shinichi} and standard techniques in rigidity theory to deduce the following ``dimension dropping'' result.

\begin{theorem}\label{theorem:rigidgloballyrigid}\cite{Jext}
      Every $(d+1)$-rigid graph is globally $d$-rigid.
\end{theorem}

We may also compare minimally $d$-rigid and minimally globally $d$-rigid graphs.
Minimally $d$-rigid graphs on $n$ vertices correspond to the bases of the $d$-dimensional generic rigidity matroid $\mathcal{R}_d(K_n)$. In particular, they are \emph{$\mathcal{R}_d$-independent}, which means that their edge set is independent in $\mathcal{R}_d(K_n)$, and for $n \geq d+1$, they have exactly $dn - \binom{d+1}{2}$ edges.  

Again, global $d$-rigidity does not have an associated matroid, and hence the structure of minimally globally $d$-rigid graphs is less understood. It follows from Hendrickson's result that a globally $d$-rigid graph on at least $d+2$ vertices cannot be minimally $d$-rigid. On the other hand, the following recent result implies that a $(d+1)$-rigid graph on at least $d+3$ vertices cannot be minimally globally $d$-rigid, simply because it has too many edges.

\begin{theorem}\cite{GJ}\label{theorem:minimallygloballyrigid}
      Let $G = (V,E)$ be a minimally globally $d$-rigid graph on at least $d+2$ vertices. Then 
      \begin{equation*}
            |E| \leq (d+1)|V| - \binom{d+2}{2},
      \end{equation*}
      where equality holds if and only if
$G=K_{d+2}$.
\end{theorem}

The first main result of this paper is the following direct connection between global rigidity in dimension $d$ and rigidity in dimension $d+1$.

\begin{theorem}\label{theorem:globallyrigidindependent}
      Every minimally globally $d$-rigid graph is $\mathcal{R}_{d+1}$-independent.
\end{theorem}

Our proof of \cref{theorem:globallyrigidindependent} is based on two main ideas. The first is to prove local analogues of \cref{theorem:shinichi,theorem:rigidgloballyrigid}. Given a vertex pair $u,v \in V$, the pair $\{u,v\}$ is \emph{$d$-linked} in $G$ (resp.\ \emph{globally $d$-linked} in $G$) if for all generic embeddings $p: V \to \R^d$, each continuous (resp.\ arbitrary) deformation of $p$ that preserves the edge lengths of $G$ also preserves the distance between $p(u)$ and $p(v)$. Thus $G$ is $d$-rigid (resp.\ globally $d$-rigid) if and only if every pair of vertices of $G$ is $d$-linked (resp.\ globally $d$-linked) in $G$. 

We prove the following strengthening of \cref{theorem:shinichi}. We say that a pair of vertices $\{u,v\}$ is \emph{vertex-redundantly $d$-linked} in $G$ if $\{u,v\}$ is $d$-linked in $G-z$ for every $z \in V - \{u,v\}$.

\begin{theorem}\label{theorem:vertexredundantlylinkedgloballylinked}
      Let $G = (V,E)$ be a graph on at least three vertices and  $u,v \in V$. If $\{u,v\}$ is vertex-redundantly $d$-linked in $G$, then $\{u,v\}$ is globally $d$-linked in $G$.  
\end{theorem}

Using standard techniques, we can turn \cref{theorem:vertexredundantlylinkedgloballylinked} into the following dimension dropping result, which was conjectured to hold in~\cite[Conjecture 4.1]{GJ}, and from which \cref{theorem:globallyrigidindependent} can be easily deduced.

\begin{theorem}\label{theorem:linkedgloballylinked}
      Let $G = (V,E)$ be a graph and  $u,v \in V$. If $\{u,v\}$ is $(d+1)$-linked in $G$, then $\{u,v\}$ is globally $d$-linked in $G$.  
\end{theorem}

The second main idea involved in proving \cref{theorem:globallyrigidindependent} (and also Theorems \ref{theorem:vertexredundantlylinkedgloballylinked} and \ref{theorem:linkedgloballylinked}) is the notion of stress-linked vertex pairs, which was  recently introduced by the first author in \cite{stresslinked}. A pair of vertices is \emph{$d$-stress-linked} in a graph $G$ if it is $d$-linked and every generic realisation of $G$ in $\R^d$ satisfies a certain linear algebraic condition (see \cref{subsection:stresslinked}). It was shown in~\cite{stresslinked} that $d$-stress-linked vertex pairs are globally $d$-linked. Instead of verifying \cref{theorem:vertexredundantlylinkedgloballylinked,theorem:linkedgloballylinked} directly, we prove  stress-linked analogues of these results (\cref{theorem:vertexredundantlylinkedstronger,theorem:linkedtostresslinked}).

Working with stress-linked vertex pairs also lets us unify and extend previous results about the sparsity of several graph classes related to rigid and globally rigid graphs. To explain this connection, let us consider the following properties of a graph $G = (V,E)$ on $n$ vertices. Below, $r_d(G)$ denotes the rank of $E$ in the rigidity matroid $\mathcal{R}_d(K_n)$.  
\begin{enumerate}
      \item $|E| \leq (d+1)n - \binom{d+2}{2}$, %
      \item $G$ is $\mathcal{R}_{d+1}$-independent,
      \item $|E| \leq r_d(G) + n - d - 1$, with equality if and only if $G$ is a copy of $K_{d+2}$. 
      \item $|E| \leq (d+1)n - {(d+1)}^2$.
\end{enumerate}

Clearly, \emph{(d)} implies \emph{(a)}. It follows from Maxwell's criterion for independence in $\cR_d(G)$, see \cite{schulze.whiteley_2017}, that \emph{(b)} and \emph{(c)} also imply \emph{(a)} when $n\geq d+2$; in fact, since a subgraph of an $\mathcal{R}_{d+1}$-independent graph is again $\mathcal{R}_{d+1}$-independent, \emph{(b)} implies that the sparsity bound \emph{(a)} holds for all subgraphs of $G$ on at least $d+2$ vertices. 

There is a variety of graph classes whose minimal members satisfy some or all of these properties. The prototypical example of this is the class of $(d+1)$-connected graphs. It is a classical result of Mader~\cite{mader_1972a} that \emph{(a)} is true for any minimally $(d+1)$-connected graph $G = (V,E)$, and that if $|V| \geq 3d+1$, then \emph{(d)} also holds. Using results from rigidity theory, it is not difficult to show that these graphs also satisfy \emph{(b)} and \emph{(c)} (see \cref{theorem:mainstronger2}).

Similarly, \emph{(b)}, and hence also \emph{(a)} when $n\geq d+2$, hold for minimally vertex-redundantly $d$-rigid graphs~\cite{kaszanitzky.kiraly_2016},  minimally (edge-)redundantly $d$-rigid graphs~\cite{Jext}, and (by \cref{theorem:globallyrigidindependent}) minimally globally $d$-rigid graphs. The latter two families (but not the first) are also known to satisfy \emph{(c)}, and it was shown in~\cite{Jear} that minimally redundantly $2$-rigid graphs on at least seven vertices satisfy  \emph{(d)} when $d=2$. 

To unify these results and to explain the gap between properties \emph{(a)} and \emph{(c)}, we introduce the following notion. A graph $G$ is \emph{$d$-stress-independent} if $\{u,v\}$ is not $d$-stress-linked in $G-uv$, for every edge $uv$ of $G$. Results in~\cite{stresslinked} imply that minimally globally $d$-rigid graphs, minimally redundantly $d$-rigid graphs, and minimally $(d+1)$-connected graphs are all $d$-stress-independent, and that \emph{(c)} holds for any $d$-stress-independent graph on at least $d+2$ vertices. We show that \emph{(b)} also holds for these graphs (\cref{theorem:mainstrongerstress}), and conjecture that they also satisfy \emph{(d)}, provided that the number of vertices is large enough (\cref{conjecture:stressindependentsharp}). We verify this conjecture for $d=2$ (\cref{theorem:2dsharpbound}). Note that for $d=1$, \emph{(c)} implies \emph{(d)} for graphs on at least $4$ vertices.

We also consider the family of {\em minimally $\mathcal{R}_d$-connected graphs} -- graphs whose $d$-dimensional rigidity matroid is connected, but does not remain so after the deletion of any edge. This family was investigated by the third author in~\cite{Jear}, where it was conjectured (and verified for $d=2$) that \emph{(a)} holds for minimally $\mathcal{R}_d$-connected graphs with $n \geq d+2$, and that \emph{(d)} holds when $n \geq \binom{d+2}{2}+1$. We partly resolve this conjecture by showing that \emph{(b)}  holds for minimally $\mathcal{R}_d$-connected graphs for all $n$ and $d$, and that \emph{(c)} holds when $n\geq d+2$. In fact, we prove a slightly stronger result that also includes the case of minimally redundantly $d$-rigid graphs (see \cref{theorem:mainstronger}). We also show that a conjecture from~\cite{stresslinked} would imply that every minimally $\mathcal{R}_d$-connected graph is $d$-stress-independent (\cref{prop:minimallyMconnectedstresslinked}). See \cref{fig:diag1} for an overview of the graph classes considered in this paper and our main results and conjectures.  

The paper is organised as follows.
In the next section we introduce some basic notions from rigidity theory, as well as recall  material from~\cite{stresslinked} about stress-linked vertex pairs. In \cref{section:stresslinked} we prove \cref{theorem:globallyrigidindependent,theorem:vertexredundantlylinkedgloballylinked,theorem:linkedgloballylinked}. In \cref{section:stressindependent} we introduce and discuss the notion of stress-independent graphs and prove \cref{theorem:mainstrongerstress,theorem:2dsharpbound}. In \cref{section:Mconnected}, we investigate minimally $\mathcal{R}_d$-connected graphs and prove \cref{theorem:mainstronger}. Finally, in \cref{section:closing} we prove an additional result about $2$-stress-independent graphs and discuss how our results fit into the literature about minimally $k$-connected graphs.

\begin{figure}[t]
      \centering
      \begin{tikzcd}[column sep={6cm,between origins},row sep={2.2cm,between origins}, cells={anchor=center}, shorten >=4pt, shorten <=4pt]
            \text{minimally globally } d\text{-rigid} \arrow[dd, "\text{\cite{stresslinked}}" description] 
            & \text{minimally redundantly } d\text{-rigid} \arrow[d] 
            & \text{minimally } \mathcal{R}_d\text{-connected} \arrow[dd]
            \\
            & \text{minimally } \mathcal{R}_d\text{-bridgeless} \arrow[dl, "\text{\cite{stresslinked}}" description] \arrow[dr] 
            &  
            \\
            d\text{-stress-independent} \arrow[dr,Rightarrow,red, sloped, "\text{\cref{theorem:mainstrongerstress}}"{yshift=1pt}] \arrow[dd,dashed,"\text{\cref{conjecture:stressindependentsharp}}" description] 
            & 
            & \parbox[c]{4cm}{\centering every $\mathcal{R}_d$-component is\\minimally $\mathcal{R}_d$-connected} \arrow[Rightarrow,red,sloped, "\text{\cref{theorem:mainstronger}}"{yshift=1pt}]{dl} \arrow[ll,dashed,"\text{\cref{conjecture:stresslinkedMcomponent} + \cref{prop:minimallyMconnectedstresslinked}}" description] 
            \\  
            & \mathcal{R}_{d+1}\text{-independent} \arrow[d] 
            & 
            \\  
            \parbox[c]{4cm}{\centering at most $(d+1)|V| - {(d+1)}^2$ \\ edges}
            \arrow[r, shorten <=-6pt, shorten >=-6pt]
            & \parbox[c]{4cm}{\centering at most $(d+1)|V| - \binom{d+1}{2}$ \\ edges} 
            & \parbox[c]{3.5cm}{\centering minimally vertex-redundantly $d$-rigid} \arrow[ul,shorten <=-8pt,"\text{\cite{kaszanitzky.kiraly_2016}}"{pos=0.44, description}] 
      \end{tikzcd}
      \caption{A diagram showing the relationships between the various graph classes discussed in this paper. %
      }\label{fig:diag1}
\end{figure}
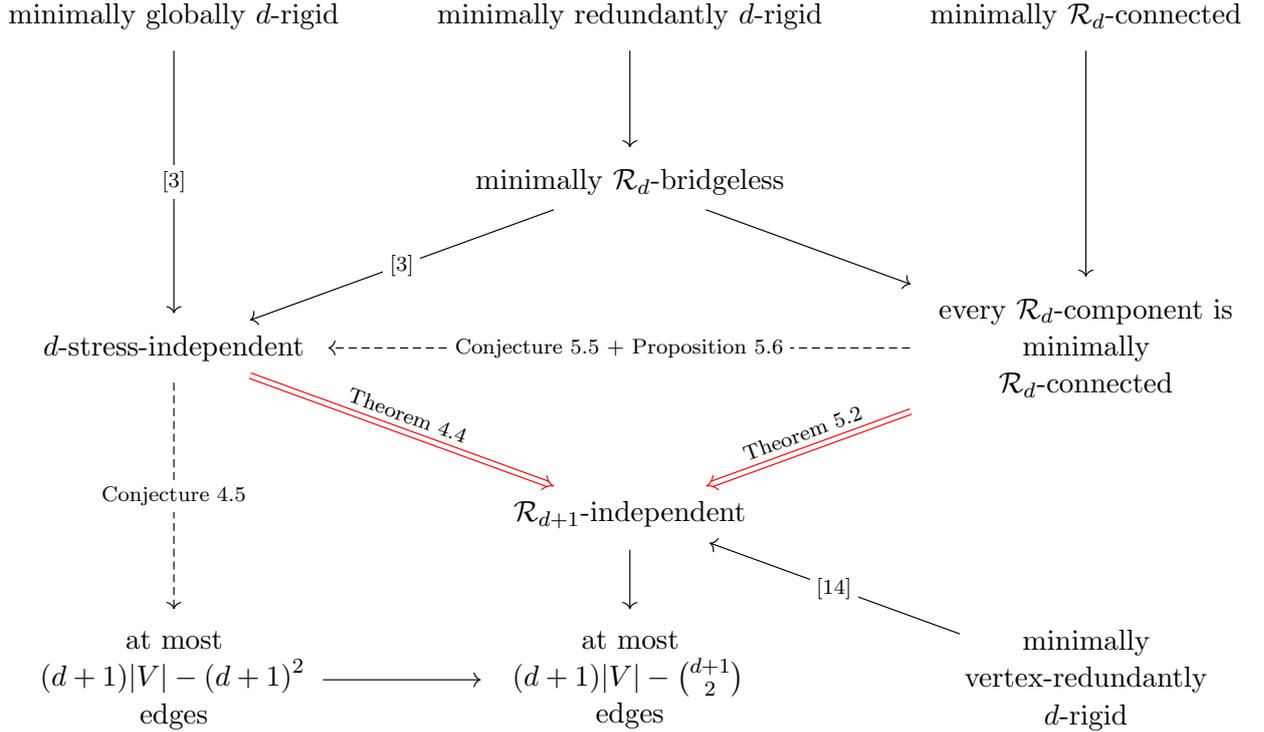

\section{Preliminaries}\label{section:prelims}

All graphs considered throughout the paper are finite and simple.
For a graph $G$, we let $V(G)$ and $E(G)$ denote the vertex and edge set of $G$, respectively. Given a pair of nonadjacent vertices $u,v$ of $G$,  $\kappa(u,v;G)$ denotes the maximum number of pairwise internally disjoint $u,v$-paths in $G$.
We will use $d$ to denote a fixed positive integer throughout the paper.

\subsection{Rigid graphs and the generic rigidity matroid}
We  recall the definitions and results from rigidity theory that we shall use. We refer the reader to  \cite{schulze.whiteley_2017,jordan.whiteley_2017} for a more detailed introduction to this topic. 

A \emph{framework} in $\R^d$ is a pair $(G,p)$ consisting of a graph $G = (V,E)$ and a mapping $p : V \to \R^d$. We will informally refer to both the framework $(G,p)$ and the map $p$ as a {\em realisation} of $G$ in $\R^d$. The framework is \emph{generic} if the multiset of coordinates of $\{p(v), v \in V\}$ is algebraically independent over $\rat$. Two frameworks $(G,p)$ and $(G,q)$ are \emph{equivalent} if $\lVert p(u) - p(v) \rVert = \lVert q(u) - q(v) \rVert$ holds for each edge $uv \in E$, and they are \emph{congruent} if the above equality holds for every pair of vertices $u,v \in V$.

To each framework $(G,p)$ we associate its \emph{rigidity matrix} $R(G,p) \in \R^{|E| \times d|V|}$, a matrix whose rows are indexed by $E$, columns are
indexed by $V \times \{1, \dots, d\}$, and in the row corresponding to the edge $uv$, the entries in the $d$ columns corresponding to vertices $u$ and $v$ contain the vectors $p(u) - p(v)$ and $p(v) - p(u)$, respectively, while the remaining entries are zeros.
The \emph{$d$-dimensional generic rigidity matroid} $\mathcal{R}_d(G)$ of $G$ is the row matroid of $R(G,p)$, for any generic framework $(G,p)$ in $\R^d$. We let $r_d$ denote the rank function of $\mathcal{R}_d(G)$, and by a slight abuse of notation, we write $r_d(G)$ for the rank of $\mathcal{R}_d(G)$.

A graph $G$ on $n$ vertices is \emph{$d$-rigid} if $r_d(G) = r_d(K_n)$. When $n \geq d$, this is equivalent to $r_d(G) = dn - \binom{d+1}{2}$. The graph $G$ is \emph{${\cal R}_d$-independent} if ${\cal R}_d(G)$ is a free matroid, and is an \emph{${\cal R}_d$-circuit} if $E$ is a circuit in ${\cal R}_d(G)$ and $G$ has no isolated vertices. If $G$ is $\mathcal{R}_d$-independent, on at least $d$ vertices, then the above mentioned result of Maxwell
implies that 
\begin{equation}\label{eq:Rdsparsity}
      |E(G)| = r_d(G) \leq d|V(G)| - \binom{d+1}{2},
\end{equation}
and since subgraphs of $\mathcal{R}_d$-independent graphs are also $\mathcal{R}_d$-independent, the bound in \cref{eq:Rdsparsity} also applies to every subgraph of $G$ on at least $d$ vertices.

An edge $e$ of $G$ is an \emph{${\cal R}_d$-bridge} if $r_d(G-e)=r_d(G)-1$ holds. Equivalently, $e$ is an ${\cal R}_d$-bridge if it is not contained in any subgraph of $G$ that is an ${\cal R}_d$-circuit. Similarly, for a pair of vertices $u,v \in V$, we say that $\{u,v\}$ is \emph{$d$-linked in $G$} if $r_d(G+uv)=r_d(G)$. Hence $\{u,v\}$ is $d$-linked in $G$ if and only if $uv$ is not an $\mathcal{R}_d$-bridge in $G+uv$. 

The graph $G$ is \emph{${\cal R}_d$-connected} if ${\cal R}_d(G)$ is connected (i.e., every pair of edges of $G$ is  contained in an ${\cal R}_d$-circuit in $G$)
and $G$ has no isolated vertices. We say that an ${\cal R}_d$-connected graph $G = (V,E)$ with $|E|\geq 2$ is \emph{minimally ${\cal R}_d$-connected} if $G-e$ is not ${\cal R}_d$-connected for all $e\in E$.

We will need the following ``dimension dropping'' result.

\begin{theorem}\label{theorem:Mconnecteddimensiondropping}\cite[Theorem 5.1]{GGJ}
      Every $\Rd$-connected graph is $\mathcal{R}_{d'}$-connected for all $1 \leq d' < d$.
\end{theorem}

The \emph{$d$-dimensional edge split} operation consists of subdividing an edge by a new vertex $w$, and then adding $d-1$ more edges incident to $w$. Given a pair of graphs $G_1$ and $G_2$ and edges $u_iv_i \in E(G_i), i \in \{1,2\}$, the \emph{parallel connection} of $G_1$ and $G_2$ along $u_1v_1$ and $u_2v_2$ is the graph obtained by taking the union of $G_1$ and $G_2$ and then identifying $u_1$ with $v_1$ and $u_2$ with $v_2$. The \emph{$2$-sum} of $G_1$ and $G_2$ along $u_1v_1$ and $u_2v_2$ is obtained by taking the parallel connection and then deleting the identified edge. We denote the $2$-sum by $G_1 \oplus_2 G_2$.

\begin{lemma}\cite[Lemma 3.9]{jackson.jordan_2005}\label{lemma:edgesplit}
      If $G$ is obtained from an $\mathcal{R}_2$-connected graph by a $2$-dimensional edge split operation, then $G$ is also $\mathcal{R}_2$-connected.
\end{lemma}

\begin{lemma}\label{lemma:Mconnected2sum}(\cite[Lemma 4.9]{GJ} and~\cite[Lemma 10]{grasegger.etal_2022})
      Let $G_i = (V_i,E_i), i \in \{1,2\}$ be two disjoint graphs and fix $u_iv_i \in E_i, i \in \{1,2\}$. The following are equivalent.
      \begin{enumerate}[label=\textit{(\alph*)}]
            \item $G_1$ and $G_2$ are both $\Rd$-connected,
            \item the 2-sum $G_1 \oplus_2 G_2$ of $G_1$  and $G_2$ along $u_1v_1$ and $u_2v_2$ is $\Rd$-connected,
            \item the parallel connection of $G_1$  and $G_2$ along $u_1v_1$ and $u_2v_2$ is $\Rd$-connected.
      \end{enumerate}
      Moreover, $G_1 \oplus_2 G_2$ is an $\mathcal{R}_d$-circuit if and only if $G_1$ and $G_2$ are both $\mathcal{R}_d$-circuits.
\end{lemma}

Given a matroid ${\cal M}=(E,r)$, and a subset $F \subseteq E$, we let ${\cal M}|_F$ denote the restriction of ${\cal M}$ to $F$. The maximal sets $F\subseteq E$ with the property that ${\cal M}|_F$ is connected are called the \emph{connected components} of ${\cal M}$. A standard result in matroid theory tells us that the connected components of ${\cal M}$ are pairwise disjoint and that the rank of ${\cal M}$ is equal to the sum of the ranks of its connected components. We refer to the subgraphs of a graph $G$  which are  (edge-)induced by the connected components of $\cR_d(G)$ as the \emph{$\cR_d$-components} of $G$. 

A \emph{cocircuit} in a matroid ${\cal M}$ is a minimal set of elements whose deletion decreases the rank of $\mathcal{M}$. In particular, a pair of elements $\{e,f\}$ is a cocircuit if and only if
$r({\cal M})=r({\cal M}-e)=r({\cal M}-f)$ and $r({\cal M}-\{e,f\})=r({\cal M})-1$ hold.

The \emph{cone} of a graph $G$, denoted by $G*x$, is obtained from $G$
by adding a new vertex $x$ and an edge from $x$ to each vertex of $G$.
A fundamental result of Whiteley~\cite{Whcone} states that
$G$ is $d$-rigid if and only if $G*x$ is $(d+1)$-rigid.

We will need the following lemma. Part \emph{(a)} is given in~\cite{kaszanitzky.kiraly_2016} and part \emph{(b)} is~\cite[Lemma 16]{K}. We  reproduce their proofs for completeness.

\begin{lemma}\label{lem:bridges} Let $G=(V,E)$ be a graph with $e,f\in E$, and let $G*x$ be the cone of $G$. Then:
      \begin{enumerate}
            \item  $e$ is an $\cR_d$-bridge  in  $G$ if and only if $e$ is an $\mathcal{R}_{d+1}$-bridge in  $G*x$;
            \item  $e, f$ are both bridges in  $\mathcal{R}_{d+1}(G)$ if $\{e,f\} $ is a cocircuit  in  $\mathcal{R}_d(G)$.
      \end{enumerate}
\end{lemma}
\begin{proof} \textit{(a)} This follows immediately from Whiteley's Coning Theorem, which implies  that $r_d((G*x)-e) = r_d(G-e) + |V|$ and $r_d(G^*) = r_d(G) + |V|$.
      
      \textit{(b)} Choose a vertex $v\in V$ incident with $e$ but not with $f$. Since $\{e,f\}$ is a cocircuit  in  $\mathcal{R}_d(G)$, $f$ is a bridge in $\mathcal{R}_d(G-e)$. This implies that $f$ is a bridge in $\mathcal{R}_d(G-v)$ and hence, by part \emph{(a)}, $f$ is a bridge in $\mathcal{R}_{d+1}((G-v)*v)$. Since $G$ is a subgraph of $(G-v)*v$, $f$ is also a bridge in $\mathcal{R}_{d+1}(G)$. By symmetry, $e$ is also a bridge in $\mathcal{R}_{d+1}(G)$.
\end{proof}

A similar argument shows that the members of a cocircuit of size $k$ in $G$ are bridges in $\cR_{d+k-1}(G)$, see~\cite[Lemma 16]{K}.

We say that a graph $G$ is \emph{redundantly $d$-rigid} if $G-e$ is $d$-rigid for every edge $e$ of $G$, and \emph{minimally redundantly $d$-rigid} if $G$ is redundantly $d$-rigid but $G-e$ is not, for every edge $e$ of $G$. More generally, $G$ is \emph{minimally $\mathcal{R}_d$-bridgeless} if $G$ has no $\mathcal{R}_d$-bridges but $G-e$ does, for every edge $e$ of $G$. In matroidal terms, this means that every edge of $G$ is contained in a size two cocircuit of $\mathcal{R}_d(G)$. Note that the minimally redundantly $d$-rigid graphs are precisely the rigid minimally $\mathcal{R}_d$-bridgeless graphs.

We say that a framework $(G,p)$ in $\R^d$ is \emph{globally rigid} if every framework in $\R^d$ that is equivalent to $(G,p)$ is congruent to $(G,p)$. The graph $G$ is \emph{globally $d$-rigid} if every generic framework of $G$ in $\R^d$ is globally rigid. 
For a pair of vertices $u,v$ of $G$, we say that $\{u,v\}$ is \emph{globally $d$-linked in $G$} if,  for every generic framework $(G,p)$ in $\R^d$ and every equivalent framework $(G,q)$, we have $\lVert p(u) - p(v)\rVert = \lVert q(u) - q(v) \rVert$.\footnote{Note that there is a fundamental difference in the properties of being $d$-linked and globally $d$-linked. Two vertices $u,v$ of a graph $G$ are linked in all generic realisations of $G$ in $\R^d$ if and only if they are linked in some generic realisation. On the other hand, it is possible that $u,v$ can be globally linked in some generic realisations of $G$ in $\R^d$ and not in others.} Thus $G$ is globally $d$-rigid if and only if every pair of vertices of $G$ is globally $d$-linked in $G$.

We will need the following combinatorial characterisation of globally $2$-rigid graphs from \cite{jackson.jordan_2005}.

\begin{theorem} \cite{jackson.jordan_2005}
\label{theorem:jacksonjordan} Let $G$ be a graph on at least $4$ vertices.
      Then following statements are equivalent:
      \begin{enumerate}
            \item $G$ is globally $2$-rigid;
            \item $G$ is $3$-connected and $\mathcal{R}_2$-connected;
            \item $G$ can be obtained from the complete graph $K_4$ by a sequence of $2$-dimensional edge split operations and edge additions.
      \end{enumerate}
\end{theorem}

We will also use the following recent result on global $d$-rigidity.

\begin{theorem}\cite{GGJ}\label{theorem:mconnected}
      Every graph  on at least $d+2$ vertices which is globally $d$-rigid is $\Rd$-connected.
\end{theorem}

\subsection{Equilibrium stresses and stress-linked vertex pairs}\label{subsection:stresslinked}

Let $G = (V,E)$ be a graph. Given a framework $(G,p)$ in $\R^d$, a vector $\omega \in \R^E$ is an \emph{(equilibrium) stress} of $(G,p)$ if the equilibrium condition
\begin{equation}\label{eq:equilibrium}
      \sum_{u: uv \in E} \omega(uv)\left(p(u)-p(v)\right) = 0
\end{equation}
holds for every $v \in V$. Equivalently, $\omega$ is a stress of $(G,p)$ if and only if $\omega$ belongs to the cokernel of the rigidity matrix $R(G,p)$.
We use $S(G,p) \subseteq \R^{E}$ to denote the space of stresses of $(G,p)$.

To every stress $\omega$, we associate its \emph{stress matrix} $\Omega \in \R^{|V| \times |V|}$, a matrix whose rows and columns are indexed by $V$ and whose entries are given by
\[\Omega_{uv} = \begin{cases}
- \omega_{uv} & \text{if } u \neq v \text{ and } uv \in E, \\ 0 & \text{if } u \neq v \text{ and } uv \notin E, \\ {\displaystyle \sum_{uw \in E}} \omega_{uw} & \text{if } u = v.
\end{cases}\]
If we represent the realisation $p$ as a matrix $P \in \R^{|V| \times d}$, then $\omega \in S(G,p)$ holds if and only if $\Omega P = 0$. More precisely, the equilibrium conditions \cref{eq:equilibrium} are satisfied for $\omega$ and $(G,p)$ at vertex $v$ if and only if the row of $\Omega P$ corresponding to $v$ is zero.

The following stress-based analogue of globally linked pairs was introduced in~\cite{stresslinked}. 
Let $G = (V,E)$ be a graph, and let $u,v \in V$ be a pair of vertices. We say that $\{u,v\}$ is \emph{$d$-stress-linked in $G$} if $\{u,v\}$ is $d$-linked in $G$ and, for every generic framework $(G,p)$ in $\R^d$ and every framework $(G,q)$ such that $S(G,p) \subseteq S(G,q)$, we have $S(G+uv,p) \subseteq S(G+uv,q)$. Note that, since $G+uv=G$ whenever $uv \in E$, adjacent pairs of vertices are always $d$-stress linked in $G$. 

It is known that $d$-stress-linked pairs are globally $d$-linked.
More precisely, parts \emph{(a)} and \emph{(c)} of the following theorem were proved in~\cite[Theorems 4.2 and 4.7]{stresslinked}. Part \emph{(b)} follows from the folklore observation that if $\{u,v\}$ is globally $d$-linked in $G$, then $\kappa(u,v;G) \geq d+1$. %

\begin{theorem}\label{theorem:stresslinked}
      Let $G$ be a graph and $u,v$ be a pair of nonadjacent vertices in $G$. Suppose that $\{u,v\}$ is $d$-stress-linked in $G$. Then
      \begin{enumerate}
            \item
            $\{u,v\}$ is globally $d$-linked in $G$,
            \item $\kappa(u,v;G) \geq d + 1$, and
            \item
            $uv$ is not contained in any size two cocircuit of $\mathcal{R}_d(G+uv)$.
      \end{enumerate}
\end{theorem}

It is conjectured in~\cite{stresslinked} that being $d$-stress-linked is, in fact, equivalent to being globally $d$-linked. This conjecture holds for the special case of globally $d$-rigid graphs by the following result.

\begin{proposition}\label{prop:globalstresslinked}\cite[Proposition 4.3]{stresslinked} A graph $G$ is globally $d$-rigid if and only if every pair of vertices is $d$-stress-linked in $G$.
\end{proposition}

We will need the following basic fact about stress-linked vertex pairs.

\begin{lemma}\label{lemma:stresslinkedsubgraph}\cite[Lemma 4.9]{stresslinked}
      Let $G$ be a graph, $G_0$ be a subgraph of $G$, and $u,v$ be a pair of vertices of $G_0$. If $\{u,v\}$ is $d$-stress-linked in $G_0$, then $\{u,v\}$ is $d$-stress-linked in $G$.
\end{lemma}

We will also use the following combinatorial characterisation of $2$-stress-linked pairs.

\begin{theorem}\cite[Theorem 4.15]{stresslinked}\label{theorem:2dstresslinked}
      Let $G = (V,E)$ be a graph, and $u,v \in V$. Then $\{u, v\}$ is 2-stress-linked in $G$ if and only if either $uv \in E$, or there is a $\mathcal{R}_2$-connected subgraph $H$ of $G$ with $u, v \in V(H)$ and $\kappa(u,v;H) \geq 3$.
\end{theorem}

We close this section by recalling some technical material about equilibrium stresses that we will need in \cref{section:stressindependent}.
Let $G$ be a graph on $n$ vertices, and fix a generic framework $(G,p)$ in $\R^d$. The \emph{$d$-dimensional shared stress nullity} of $G$ is the number \[\sharedstressnullity{d}{G} = \dim\left(\bigcap\big\{\ker(\Omega): \Omega \text{ is a stress matrix of } (G,p)\big\}\right)\]
It is known that this number is independent of the (generic) choice of $(G,p)$.
For convenience, let us also introduce the notation \[\sharedstressrank{d}{G} = n - \sharedstressnullity{d}{G}.\]
It is not difficult to see that if $\{u,v\}$ is $d$-stress linked in $G$, then $\sharedstressnullity{d}{G}=\sharedstressnullity{d}{G+uv}$ and $\sharedstressrank{d}{G}=\sharedstressrank{d}{G+uv}$ (see~\cite[Proposition 4.1]{stresslinked}).
It is also known that if $G$ has at least $d+1$ vertices, then 
\begin{equation*} d+1 \leq \sharedstressnullity{d}{G} \leq n \hspace{1em} \text{ and hence } \hspace{1em} 0 \leq \sharedstressrank{d}{G} \leq n-d-1.\end{equation*}
By a seminal result of Gortler, Healy and Thurston, $G$ is globally $d$-rigid if and only if $\sharedstressrank{d}{G}$ achieves this upper bound.

\begin{theorem}\label{theorem:ght}\cite[Theorem 4.4]{gortler.etal_2010}
      Let $G$ be a graph on  $n\geq d+2$ vertices. Then $G$ is globally $d$-rigid if and only if $\sharedstressrank{d}{G} = n - d - 1$.
\end{theorem}

Let us also note that $\sharedstressrank{d}{G} = 0$ if and only if $G$ is $\mathcal{R}_d$-independent, and that adding isolated vertices or $\mathcal{R}_d$-bridges to $G$ does not change the value of $\sharedstressrank{d}{G}$. Finally, note that if $G$ is an $\mathcal{R}_d$-circuit, then each generic framework $(G,p)$ has a unique nonzero stress matrix $\Omega$ up to scaling, and in this case we have $\sharedstressnullity{d}{G} = n - \rk(\Omega)$ and $\sharedstressrank{d}{G} = \rk(\Omega)$.

\section{A sufficient condition for a pair of vertices to be stress-linked}\label{section:stresslinked}

In this section we deduce \cref{theorem:vertexredundantlylinkedgloballylinked,theorem:linkedgloballylinked} from their stress linked versions \cref{theorem:vertexredundantlylinkedstronger,theorem:linkedtostresslinked}, respectively.
The main technical ingredient of our proofs is \cref{lemma:almoststress} below. This lemma also appears as~\cite[Lemma 3.12]{stresslinked}; since the proof is short, we repeat it here for completeness. 
We will need the following folklore result from linear algebra.

\begin{lemma}\label{lemma:galeduality}
      Let $A \in \R^{m \times n}$ and $B \in \R^{n \times \ell}$ be matrices such that $AB = 0$. Fix $S \subseteq \{1,\ldots,n\}$, and define $\hat{S} = \{1,\ldots,n\} - S$. If the rows of $B$ indexed by $\hat{S}$ are linearly independent, then the columns of $A$ indexed by $S$ span the column space of $A$. 
\end{lemma}
\begin{proof}
      The validity of the statement does not change by reordering the columns of $A$ and the rows of $B$ according to the same permutation. Thus we may assume that $S = \{1,\ldots,k\}$ for some $k$. By performing column operations on $B$, as well as possibly deleting some columns, we may also assume that $B$ is of the form 
      \[B = \begin{pmatrix}
            B' \\ I
      \end{pmatrix},\]
      where $I$ is an $(n-k) \times (n-k)$ identity matrix. 
      Since $AB=0$, we can use the $i$-th column of $B$ to express the $(k+i)$-th column of $A$ as a linear combination of its first $k$ columns, for each $i \in \{1,\ldots,n-k\}$. This shows that the first $k$ columns of $A$ indeed span its column space. 
\end{proof}

\begin{lemma}\label{lemma:almoststress}
      Let $G = (V,E)$ be a graph, and let $(G,p)$ and $(G,q)$ be frameworks in $\R^d$ such that $(G,p)$ is in general position. Fix $\omega \in S(G,p)$, and let $U \subseteq V$ be the set of vertices where the equilibrium conditions \cref{eq:equilibrium} are not satisfied for $\omega$ and $(G,q)$. If $U\neq \varnothing$, then  $|U| \geq d+2$.
\end{lemma}
\begin{proof} Suppose, for a contradiction that $1\leq |U| \leq d+1$.
      Let $\Omega \in \R^{V \times V}$ be the stress matrix of $\omega$, let $P,Q \in \R^{|V| \times d}$ be the matrices whose rows are the vectors $p(v), v \in V$ and $q(v), v \in V$, respectively, and let $\widetilde{P}$ and $\widetilde{Q}$ be obtained from $P$ and $Q$ by adding a column of ones at the end.
      
      The row of $\Omega \widetilde{Q}$ indexed by a vertex $v \in V$ is zero if and only if the equilibrium conditions \cref{eq:equilibrium} are satisfied at $v$ for $\omega$ and $(G,q)$. Hence the definition of $U$ implies that $(\Omega \widetilde{Q})_v = 0$ for all $v \in V - U$ and we can obtain the required contradiction  by showing that $\Omega \widetilde{Q} = 0$. To this end, it will suffice to show that the rows of $\Omega$ corresponding to $V - U$ span the  row space of $\Omega$. Since $\Omega$ is symmetric, this is equivalent to showing that the columns corresponding to $V - U$ span the column space of $\Omega$. Since $\Omega\widetilde{P}=0$, \cref{lemma:galeduality} implies that it will suffice to show that the rows of $\widetilde{P}$ corresponding to $U$ are linearly independent. This is equivalent to $\{p(u), u \in U\}$ being affinely independent, which holds by our assumption that $(G,p)$ is in general position and $|U| \leq d+1$.
\end{proof}

We are now ready to prove the main result of this section.

\begin{theorem}\label{theorem:vertexredundantlylinkedstronger}
      Let $G = (V,E)$ be a graph on at least $d+2$ vertices, $u,v \in V$  and put \[U = \{z \in V - \{u,v\} : \{u,v\} \text{ is not } d\text{-linked in } G-z\}.\]
      If $|U| \leq d-1$, then $\{u,v\}$ is $d$-stress-linked in $G$.
\end{theorem}
\begin{proof}
      If $uv \in E$, then the statement is trivial, so let us assume that $u$ and $v$ are nonadjacent in $G$. Note that since $|V| \geq d+2$, and $|U| \leq d-1$, there is at least one vertex $z$ for which $\{u,v\}$ is $d$-linked in $G-z$, and in particular, $\{u,v\}$ is $d$-linked in $G$. Let $(G,p)$ be a generic framework of $G$ in $\R^d$ and let $(G,q)$ be a framework in $\R^d$ with $S(G,p) \subseteq S(G,q)$. By a slight abuse of notation, let us view $S(G,p)$ and $S(G,q)$ as subsets of $S(G+uv,p)$ and $S(G+uv,q)$, respectively.
      
      Our goal is to show that $S(G+uv,p) \subseteq S(G+uv,q)$. To this end, let us fix a stress $\omega \in S(G+uv,p)$. If $\omega(uv) = 0$, then $\omega \in S(G,p) \subseteq S(G+uv,q)$. Hence let us assume that $\omega(uv) \neq 0$; by scaling, we may suppose that $\omega(uv) = 1$. 
      We will show that the equilibrium conditions \cref{eq:equilibrium} hold for $\omega$ and $(G+uv,q)$ at each vertex $z \in V \sm (U\cup \{u,v\})$. It will then follow from \cref{lemma:almoststress} that $\omega \in S(G+uv,q)$.
      
      Let us thus fix a vertex $z \in V \sm (U\cup \{u,v\})$. Since $\{u,v\}$ is $d$-linked in $G-z$, there exists an $\mathcal{R}_d$-circuit $C$ in $G-z+uv$ with $uv \in E(C)$. Hence there is a (unique) nonzero equilibrium stress $\omega_C$ of $(G+uv,p)$ that is supported on $E(C)$ and satisfies $\omega_C(uv) = 1$.
      
      Let us consider $\omega' = \omega - \omega_C$. Since $\omega'(uv) = 0$, we have $\omega' \in S(G,p) \subseteq S(G,q)$. Moreover, since $\omega_C$ is identically zero on the set  $\delta(z)$ of all edges of $G$ which are incident to $z$, we have $\omega'|_{\delta(z)} = \omega|_{\delta(z)}$. Hence the equilibrium conditions \cref{eq:equilibrium} at vertex $z$ hold for  $\omega$ and $(G,q)$ if and only if they hold for $\omega'$ and $(G,q)$. Since the latter is a stress of $(G,q)$, the conditions for $\omega'$ are satisfied at $z$, and hence so are the conditions for $\omega$, as required.   
\end{proof}

\cref{theorem:vertexredundantlylinkedgloballylinked} follows immediately from \cref{theorem:stresslinked}(a) and \cref{theorem:vertexredundantlylinkedstronger}. By recalling that a graph is $d$-rigid (resp.\ globally $d$-rigid) if and only if every pair of its vertices is $d$-linked (resp.\ $d$-stress-linked), we can also deduce the following strengthening of \cref{theorem:shinichi}.

\begin{theorem}\label{corollary:shinichistronger}
      Let $G = (V,E)$ be a graph on at least $d+2$ vertices and define \[U = \{v \in V: G-v \text{ is not } d\text{-rigid}\}\]
      If $|U| \leq d-1$, then $G$ is globally $d$-rigid. 
\end{theorem}

Let us note that \cref{corollary:shinichistronger} can also be obtained by carefully following the proof of \cref{theorem:shinichi} in~\cite{tanigawa_2015}.

To deduce \cref{theorem:linkedgloballylinked} from \cref{theorem:vertexredundantlylinkedstronger}, we need the following easy lemma.

\begin{lemma}\label{lemma:d+1linked}
      Let $G = (V,E)$ be a graph, and $u,v \in V$ be a pair of vertices which  is $\mathcal{R}_{d+1}$-linked in $G$. Then $\{u,v\}$ is vertex-redundantly $d$-linked in $G$.
\end{lemma}
\begin{proof}
      Fix $z \in V - \{u,v\}$, and consider the cone graph $G_z=(G-z)*z$. 
      Since $\{u,v\}$ is $\mathcal{R}_{d+1}$-linked in $G$, it is also $\mathcal{R}_{d+1}$-linked in $G_z$.
      \cref{lem:bridges}(a) now implies that $\{u,v\}$ is $d$-linked in $G-z$.
\end{proof}

By combining \cref{theorem:vertexredundantlylinkedstronger} and \cref{lemma:d+1linked}, we obtain the following result, which immediately implies \cref{theorem:linkedgloballylinked}.

\begin{theorem}\label{theorem:linkedtostresslinked}
      Let $G = (V,E)$ be a graph, and let $u,v \in V$ be a pair of vertices. If $\{u,v\}$ is $(d+1)$-linked in $G$, then $\{u,v\}$ is $d$-stress-linked in $G$.  
\end{theorem}

To close this section, we note that \cref{theorem:vertexredundantlylinkedstronger} implies the following ``gluing theorem'' from~\cite{stresslinked}.

\begin{theorem}\label{theorem:gluing}\cite[Theorem 4.12]{stresslinked}
      Let $G = (V,E)$ be the union of the graphs $G_1 = (V_1,E_1), G_2 = (V_2,E_2)$ with $|V_1 \cap V_2| \leq d+1,$ and let $u,v \in V_1 \cap V_2$ be a pair of vertices such that $\{u,v\}$ is $d$-linked in both $G_1$ and $G_2$. Then $\{u,v\}$ is $d$-stress-linked in $G$.
\end{theorem}
\begin{proof}
      The hypothesis that $\{u,v\}$ is $d$-linked in both $G_1$ and $G_2$ implies that $\{u,v\}$ is $d$-linked in $G-z$ for all $z \in V - (V_1 \cap V_2)$. Since $|(V_1 \cap V_2) - \{u,v\}| \leq d-1$, the theorem now follows from \cref{theorem:vertexredundantlylinkedstronger}.
\end{proof}

We note that the proof of \cref{theorem:gluing} in~\cite{stresslinked} is also based on \cref{lemma:almoststress} and uses similar ideas as our proof of \cref{theorem:vertexredundantlylinkedstronger}.

\section{Stress-independent graphs}\label{section:stressindependent}

Recall from the Introduction that a graph $G$ is \emph{$d$-stress-independent} if, for every edge $uv$ of $G$, $\{u,v\}$ is not $d$-stress-linked in $G-uv$. Our main motivation for studying this graph class is given by the following proposition, which is an immediate consequence of \cref{theorem:stresslinked}.
\begin{proposition}\label{cor:d-stress}
      If $G$ is minimally $(d+1)$-connected, minimally globally $d$-rigid, or minimally $\mathcal{R}_d$-bridgeless, then $G$ is $d$-stress-independent.
\end{proposition}

Not every $d$-stress-independent graph belongs to one of the graph classes in \cref{cor:d-stress}. 
However, as a partial converse to \cref{cor:d-stress}, we have the following result.

\begin{lemma}\label{lem:new} 
      Let $G$ be a globally $d$-rigid graph. Then $G$ is $d$-stress-independent if and only if it is minimally globally $d$-rigid.
\end{lemma} 
\begin{proof} 
      \cref{prop:globalstresslinked}  tells us that every pair of vertices of $G$ is $d$-stress-linked. Thus $G-uv$ is globally $d$-rigid for some $uv\in E$ if and only if $\{u,v\}$ is $d$-stress-linked in $G-uv$. Hence, $G$ is $d$-stress-independent if and only if $G-uv$ is not globally $d$-rigid for all $uv\in E$. 
\end{proof}

Let us also note the following basic properties of $d$-stress-independent graphs that follow from \cref{lemma:stresslinkedsubgraph} and \cref{theorem:stresslinked}(c).

\begin{lemma}\label{lemma:stressindependentoperations}
      Let $G$ be a $d$-stress-independent graph. Then
      \begin{enumerate}
            \item every subgraph of $G$ is $d$-stress-independent, and
            \item if $\{u,v\} \subseteq V(G)$ is not $d$-linked in $G$, then $G+uv$ is $d$-stress-independent.
      \end{enumerate}
\end{lemma}

The notion of $d$-stress-independent graphs was first considered in~\cite{stresslinked}. In particular, it was shown in \cite[Theorem 5.2]{stresslinked} that they satisfy condition \emph{(c)} of the Introduction (condition \emph{(b)} of \cref{theorem:mainstrongerstress} below). We can give a quick proof of this result in the following strong form using \cref{theorem:linkedtostresslinked}.

\begin{theorem}\label{theorem:mainstrongerstress}
      Let $G$ be a $d$-stress-independent graph. Then:
      \begin{enumerate}
            \item $G$ is $\mathcal{R}_{d+1}$-independent;
            \item 
            for all subgraphs $G'=(V',E')$  of $G$ with  $|V'|\geq d+2$, we have \[|E'|\leq r_d(G') + |V'| - d - 1 \leq (d+1)|V'|-\binom{d+2}{2},\] 
            and the first inequality holds with equality if and only if $G'=K_{d+2}$. 
      \end{enumerate}
\end{theorem}
\begin{proof}
      We prove part \emph{(a)} by showing that every edge $uv$ of $G$ is an $\mathcal{R}_{d+1}$-bridge in $G$. This follows by applying \cref{theorem:linkedtostresslinked} to $G-uv$ and using the fact that $\{u,v\}$ is not $d$-stress linked in $G-uv$ since $G$ is $d$-stress-independent. 
      
      For part \emph{(b)}, let us add a set $F$ of $\mathcal{R}_d$-bridges to $G'$ so that $G'+F$ is $d$-rigid. By \cref{lemma:stressindependentoperations}(b), $G'+F$ is $d$-stress-independent, and hence by part \emph{(a)}, it is $\mathcal{R}_{d+1}$-independent. A simple calculation now gives
      \begin{align*}
            |E'| = |E(G'+F)| - |F| &\leq (d+1)|V'|-\binom{d+2}{2} - |F| \\ &= r_d(G'+F) + |V'| - d - 1 - |F| \\ &= r_d(G') + |V'| - d - 1, 
      \end{align*}
      as claimed. If equality holds, then $G'+F$ is minimally $(d+1)$-rigid, and hence globally $d$-rigid by \cref{theorem:rigidgloballyrigid}. Since globally $d$-rigid graphs are $\mathcal{R}_d$-bridgeless by Hendrickson's theorem (or alternatively, by \cref{theorem:mconnected}), we must have $F = \varnothing$. Thus $G'$ is globally $d$-rigid and $d$-stress-independent, so \cref{lem:new} tells us that $G'$ is minimally globally $d$-rigid. We can now apply \cref{theorem:minimallygloballyrigid} to deduce that $G'=K_{d+2}$.
\end{proof}

Since minimally globally $d$-rigid graphs are $d$-stress-independent by \cref{cor:d-stress}, \cref{theorem:globallyrigidindependent} follows immediately from \cref{theorem:mainstrongerstress}. 

We conjecture that a stronger upper bound on the number of edges 
holds for $d$-stress-independent graphs with sufficiently many vertices. 

\begin{conjecture}\label{conjecture:stressindependentsharp}
      Every $d$-stress-independent graph $G = (V,E)$ 
      with sufficiently many vertices satisfies
      \[|E| \leq (d+1)|V| - {(d+1)}^2.\]
\end{conjecture}

It is known that the complete bipartite graph $K_{d+1,n-d-1}$ is minimally globally $d$-rigid, and hence $d$-stress-independent, whenever $n \geq \binom{d+2}{2}+1$. This shows that the bound in \cref{conjecture:stressindependentsharp} would be best possible. We can also show that the number of vertices required in \cref{conjecture:stressindependentsharp} is at least $\binom{d+2}{2}+1$. Indeed, it is known that for $n = \binom{d+2}{2}$, $K_{d+1,n-d-1}$ is minimally $d$-rigid, and it follows from \cref{theorem:stresslinked}(c) that the graph obtained by adding an edge to this complete bipartite graph is $d$-stress-independent and has $(d+1)n - (d+1)^2 + 1$ edges. We believe that the bound in \cref{conjecture:stressindependentsharp} holds for any $\mathcal{R}_d$-connected $d$-stress-independent graph on at least $\binom{d+2}{2}+1$ vertices. However, in the non-$\mathcal{R}_d$-connected case this is not true (see the $d=2$ case below), and it is unclear what the bound on the number of vertices should be in this case.  

The $d=1$ case  of \cref{conjecture:stressindependentsharp} follows from  \cref{theorem:mainstrongerstress}(b).  In the remainder of this section, we use results and ideas from~\cite{stresslinked, GJ, Jear} to show that the $d=2$ case also holds.

\begin{theorem}\label{theorem:2dsharpbound}
      If $G = (V,E)$ is a $2$-stress-independent graph with $|V| \geq 8$, then $|E| \leq 3|V| - 9$.
\end{theorem}

The bound $|V| \geq 8$ is sharp: the graph obtained by gluing two copies of $K_4$ along a vertex and adding an edge is $2$-stress-independent with $7$ vertices and $13 = 3\cdot 7 - 8$ edges. 

Before proving \cref{theorem:2dsharpbound}, let us first show how it implies the following stronger form.

\begin{theorem}\label{theorem:2dsharpboundstrong}
      Let $G$ be a $2$-stress-independent graph. For all subgraphs $G'=(V',E')$  of $G$ with  $|V'|\geq 8$, we have 
      $|E'|\leq r_2(G') + |V'| - 6 \leq 3|V'| - 9.$  
\end{theorem}
\begin{proof}
      Let us add a set $F$ of $\mathcal{R}_2$-bridges to $G'$ so that $G' + F$ is $2$-rigid. By \cref{lemma:stressindependentoperations}, $G'+F$ is also $2$-stress-independent, and hence \cref{theorem:2dsharpbound} implies 
      \begin{align*}
            |E(G')| &= |E(G'+F)| - |F| \\ &\leq 3|V'| - 9 - |F| = r_2(G'+F) + |V'| - 6 - |F| = r_2(G') + |V'| - 6,
      \end{align*}
      as claimed.
\end{proof}

\cref{theorem:2dsharpbound} is new even in the special case of minimally globally $2$-rigid graphs. The case of minimally redundantly $2$-rigid graphs was shown in~\cite{Jear}, but the stronger form \cref{theorem:2dsharpboundstrong} is new in this case as well. In fact, \cref{theorem:2dsharpboundstrong} appears to be new even for minimally $3$-connected graphs. See \cref{subsection:connectivity} for a discussion of how our results fit into the literature on minimally $k$-connected graphs.

We now start working towards the proof of \cref{theorem:2dsharpbound}.
We will use the following recent result of the third author.

\begin{theorem}\label{theorem:R2connectedsharp} \cite[Theorem 1.2]{Jear}
      If $G = (V,E)$ is a minimally $\mathcal{R}_2$-connected graph with $|V| \geq 7$, then $|E| \leq 3|V| - 9$.
\end{theorem}

Let us note that by \cref{theorem:2dstresslinked}, minimally $\mathcal{R}_2$-connected graphs are $2$-stress-independent, and hence \cref{theorem:R2connectedsharp} can be seen as a special case of \cref{theorem:2dsharpbound} (with a slightly weaker hypothesis on the number of vertices).

The proof of \cref{theorem:R2connectedsharp} in~\cite{Jear} hinges on the following combinatorial result about $\mathcal{R}_2$-connected graphs which we will also need.
Let $q$ be a positive integer. We denote by $\hat{K}_4^q$ the graph obtained by taking $q$ copies of $K_4$, fixing a pair of vertices from each, and gluing them along these pairs of vertices. We let $K_4^q$ be the graph obtained from $\hat{K}_4^q$ by deleting the edge connecting the separating vertex pair.%

\begin{theorem}\label{theorem:largecircuit}\cite[Theorem 1.3]{Jear}
      Let $G$ be an $\mathcal{R}_2$-connected graph on at least seven vertices. Then either $G$ contains an $\mathcal{R}_2$-circuit with at least seven vertices, or $G = K_4^q$ or $\hat{K}_4^q$ for some $q \geq 3$. 
\end{theorem}

Next, we prove a sequence of technical lemmas. The first is a simple analogue of \cref{theorem:largecircuit}. Let $W_5$ denote the wheel graph on five vertices. It is well-known that this graph is a minimally globally $2$-rigid $\mathcal{R}_2$-circuit.

\begin{lemma}\label{lemma:fivevertices}
      If $G$ is an $\mathcal{R}_2$-connected graph on at least five vertices, then $G$ contains an $\mathcal{R}_2$-circuit on at least five vertices.
\end{lemma}
\begin{proof}
      Suppose, for a contradiction, that every $\mathcal{R}_2$-circuit in $G$ is a copy of $K_4$. Since $G$ is $\mathcal{R}_2$-connected, this implies that any pair of (non-isolated) vertices of $G$ is contained in a clique, and hence are adjacent. Thus $G$ is a complete graph. Since $G$ has at least five vertices, it contains $W_5$ as an $\mathcal{R}_2$-circuit, a contradiction.
\end{proof}

\begin{lemma}\label{lemma:sixvertices}
      Every minimally globally $2$-rigid graph on six vertices is an $\mathcal{R}_2$-circuit.
\end{lemma}
\begin{proof}
      By \cref{theorem:jacksonjordan}, every globally $2$-rigid graph on at least four vertices arises from $K_4$ by a sequence of $2$-dimensional edge split operations and edge additions. If only edge splits are used, then the graph obtained is an $\mathcal{R}_2$-circuit. If an edge addition is used, then to obtain a minimally globally $2$-rigid graph at the end of the sequence, we have to later perform an edge split along it. By combining these observations, we conclude that the only graphs we need to check are those obtained from $W_5$ by adding a new edge and then performing an edge split along it. Up to isomorphism, there are two such graphs (depending on whether the center of the wheel is a neighbour or not of the vertex we add last), and it is easy to check that neither of them is minimally globally $2$-rigid.
\end{proof}

Recall the notation $\sharedstressrank{d}{G} = |V(G)| - \sharedstressnullity{d}{G}$, where $\sharedstressnullity{d}{G}$ denotes the $d$-dimensional shared stress nullity of $G$. As we will see, for a $d$-stress-independent graph $G$, the existence of an $\mathcal{R}_d$-circuit $C$ in $G$ such that $\sharedstressrank{d}{C}$ is large allows us to give an upper bound on the number of edges in $G$.  The goal of the next two results is to characterise the $\mathcal{R}_2$-circuits $C$ with $\sharedstressrank{2}{G} \leq 2$.

\begin{lemma}\label{lemma:2sum}
      Let $C = (V,E)$ be an $\mathcal{R}_d$-circuit which is a $2$-sum of two $\mathcal{R}_d$-circuits $C_1$ and $C_2$. Then we have $\sharedstressrank{d}{C} = \sharedstressrank{d}{C_1} + \sharedstressrank{d}{C_2}$.
\end{lemma}
\begin{proof}
      Let $V(C_1) \cap V(C_2) = \{u.v\}$, and let us fix a generic framework $(C,p)$ in $\R^d$. Since $C_i$ is an $\mathcal{R}_d$-circuit, $(C+uv,p)$ has a unique nonzero stress $\omega_i$ supported on the edges of $C_i$ and satisfying $\omega_i(uv) = 1$, for $i \in \{1,2\}$. Moreover, $\omega = \omega_1 - \omega_2$ is the unique nonzero stress of $(C,p)$, up to scalar multiplication.

      Let $\Omega_1,\Omega_2$ and $\Omega$ be the stress matrices associated to $\omega_1,\omega_2$ and $\omega$, respectively. Since $G_i$ is an $\mathcal{R}_d$-circuit, we have $\sharedstressrank{d}{C_i} = \rk(\Omega_i), i \in \{1,2\}$, and similarly, $\sharedstressrank{d}{C} = \rk(\Omega)$. Hence we only need to show that $\rk(\Omega) = \rk(\Omega_1) + \rk(\Omega_2)$. From $\Omega = \Omega_1 - \Omega_2$ we have $\rk(\Omega) \leq \rk(\Omega_1) + \rk(\Omega_2)$. To show that $\rk(\Omega) \geq \rk(\Omega_1) + \rk(\Omega_2)$ also holds, consider the submatrix $\Omega'_i$ of $\Omega_i$ formed by the rows and columns indexed by $V(C_i) - \{u,v\}$ for $i \in \{1,2\}$. Since $\Omega$ contains the block matrix 
      \[\begin{pmatrix}
            \Omega'_1 & 0 \\ 0 & \Omega'_2
      \end{pmatrix}\]
      as a submatrix, we have $\rk(\Omega) \geq \rk(\Omega'_1) + \rk(\Omega'_2)$. Hence it suffices to show that $\rk(\Omega'_i) = \rk(\Omega_i)$ for $i \in \{1,2\}$; by symmetry, we may assume that $i=1$.

      We proceed as in the proof of \cref{lemma:almoststress}. Let $P_1 \in \R^{|V(C_1)| \times d}$ be the matrix whose rows are the vectors $p(z), z \in V(C_1)$, and let $\widetilde{P_1}$ be the matrix obtained from $P_1$ by adding a column of ones at the end. Observe that the rows and columns of $\Omega_1$ corresponding to $V(C_2) - \{u,v\}$ are all zero, and hence by deleting them, the rank of $\Omega_1$ does not change. In this way, we may consider $\Omega_1$ as a $|V(C_1)| \times |V(C_1)|$ matrix. Then $\Omega_1 \widetilde{P}_1 = 0$, and the rows of $\widetilde{P}_1$ indexed by $\{u,v\}$ are linearly independent, since $\{p(u),p(v)\}$ is affinely independent. \cref{lemma:galeduality} now implies that $\rk(\Omega_1) = \rk(\Omega_1'')$, where $\Omega_1''$ is the submatrix of $\Omega_1$ obtained by deleting the columns corresponding to $p(u)$ and $p(v)$. Since $\Omega_1$ is symmetric, we also have $(\Omega_1'')^T \widetilde{P}_1 = 0$, and by using \cref{lemma:galeduality} again we deduce that \[\rk(\Omega_1') = \rk(\Omega_1'') = \rk(\Omega_1),\]
      as claimed.
\end{proof}

Let $K_4 \oplus_2 K_4$ denote the $2$-sum of two copies of $K_4$.

\begin{lemma}\label{lemma:smallcircuits}
      Let $C = (V,E)$ be an $\mathcal{R}_2$-circuit. We have
      \begin{enumerate}
            \item $\sharedstressrank{2}{C} = 1$ if and only if $C = K_4$, and
            \item $\sharedstressrank{2}{C} = 2$ if and only if $C = W_5$ or $C = K_4 \oplus_2 K_4$.
      \end{enumerate}
      In particular, if $|V| \geq 7$, then $\sharedstressrank{2}{C} \geq 3$.
\end{lemma}
\begin{proof}
      It follows from \cref{theorem:jacksonjordan} and \cref{lemma:Mconnected2sum} that $C$ is either globally $2$-rigid, or arises as the $2$-sum of two $\mathcal{R}_2$-circuits $C_1$ and $C_2$. If $C$ is globally $2$-rigid, then we have $\sharedstressrank{2}{C} = |V| - 3$ by \cref{theorem:ght}. Hence in this case $\sharedstressrank{2}{C} = 1$ if and only if $|V| = 4$, which means that $C = K_4$, and similarly, $\sharedstressrank{2}{C} = 2$ if and only if $|V| = 5$, which means that $C = W_5$. If $C$ is the $2$-sum of two $\mathcal{R}_2$-circuits $C_1$ and $C_2$, then by \cref{lemma:2sum}, we have $\sharedstressrank{2}{C} \geq 2$ with equality if and only if $\sharedstressrank{2}{C_1} = \sharedstressrank{2}{C_2} = 1$, which means that $C_1 = C_2 = K_4$.
\end{proof}

As a key step in our proof of \cref{theorem:2dsharpbound}, we now characterise the $2$-stress-independent graphs that satisfy $|E| = 3|V| - 7$. Recall that $W_5$ denotes the wheel on five vertices, and let $K_4^+$ be the graph obtained from $K_4$ by adding a vertex of degree two. 

\begin{proposition}\label{lemma:3n-7case}
      Let $G = (V,E)$ be a $2$-stress-independent graph on at least four vertices. If $|E| = 3|V| - 7$, then $G \in \{K_4-e, K_4^+, W_5\}$.
\end{proposition}
\begin{proof}
      We proceed as in the proof of \cref{theorem:mainstrongerstress}(b) in~\cite{stresslinked}. 
      Let $C$ be an arbitrary $\mathcal{R}_2$-circuit  in $G$, $e_1$ be an edge of $C$, and $G_0$ be obtained by extending $C-e_1$ to a maximal $\mathcal{R}_2$-independent subgraph $G_0=(V_0,E_0)$ of $G$. Let $E - E_0 = \{e_1,\ldots,e_k\}$, and let $G_i = G_0 + \{e_1,\ldots,e_i\}$ for each $i \in \{1,\ldots,k\}$. The maximal choice of $G_0$ and \cref{lemma:stresslinkedsubgraph} imply that the pair of end vertices $\{u_i,v_i\}$ of $e_i$ is $\mathcal{R}_2$-linked but not $2$-stress-linked in $G_{i-1}$, for each $i \in \{1,\ldots,k\}$. 
      Note that since $G_1$ can be obtained from $C$ by adding isolated vertices and $\mathcal{R}_d$-bridges, $\sharedstressrank{2}{C} = \sharedstressrank{2}{G_1}$ holds. 
      Hence we have
$
      \sharedstressrank{2}{C} =\sharedstressrank{2}{G_1} < \cdots <  \sharedstressrank{2}{G_{k-1}} <  \sharedstressrank{2}{G_k} =  \sharedstressrank{2}{G}$ so $k-1 \leq \sharedstressrank{2}{G}-\sharedstressrank{2}{C}$. This gives
      \begin{equation}\label{eq:1}
            3|V|-7=|E|=|E_0|+k \leq 2|V| -3  + \sharedstressrank{2}{G} - \sharedstressrank{2}{C} + 1
      \end{equation}
      Consider the following two cases.
      
      \vspace{.5em}
      \noindent \textbf{\sffamily\boldmath Case 1: $G$ is globally $2$-rigid.} Then $\sharedstressrank{2}{G}=|V|-3$ and \cref{eq:1} gives $\sharedstressrank{2}{C}\leq 2$ for all  $\cR_2$-circuits $C$ in $G$. In addition, $G$ is $\mathcal{R}_2$-connected by
      \cref{theorem:mconnected} and $G$ is minimally globally 2-rigid by \cref{lem:new}.  
      If $|V| \geq 7$, \cref{theorem:largecircuit} and the hypothesis that 
$|E|=3|V|-7$ tell us that $G$ contains an $\cR_2$-circuit $C$ with at least seven vertices. \cref{lemma:smallcircuits} now gives  $\sharedstressrank{2}{C}\geq 3$, contradicting the fact that 
$\sharedstressrank{2}{C}\leq 2$. 
      If $|V| = 6$, then
      \cref{lemma:sixvertices} and the fact that $G$ is minimally globally 2-rigid imply  that $G$ is an $\cR_2$-circuit.  This contradicts the  hypothesis that $|E|=3|V|-7=11$ and the assumption that $|V|=6$ and hence $r_2(G) = 9$. 
      
      Hence we may assume that $|V|\leq 5$. Since $G$ is globally rigid and $|E|=3|V|-7$, we have $|V|=5$ and $G$ is an $\cR_2$-circuit. This implies that $G=W_5$.
      
      \vspace{.5em}
      \noindent \textbf{\sffamily \boldmath Case 2: $G$ is not globally $2$-rigid.}
      Then $\sharedstressrank{2}{G} \leq |V|-4$, and thus \eqref{eq:1} gives %
$\sharedstressrank{2}{C} \leq 1$. \cref{lemma:smallcircuits} now implies  that  every $\mathcal{R}_2$-circuit in $G$ is a copy of $K_4$.  By \cref{lemma:fivevertices}, this implies that every nontrivial $\mathcal{R}_2$-component of $G$ is a copy of $K_4$.
      
      Let $q$ be the number of nontrivial $\mathcal{R}_2$-components of $G$. Since the $\mathcal{R}_2$-components are edge-disjoint and each is a copy of $K_4$, we have \[6q \leq |E| = 3|V| - 7 \leq 2|V| - 3 + q,\]
      where the last inequality follows from the observation that by deleting one edge from each copy of $K_4$, we obtain an $\mathcal{R}_2$-independent graph. This yields
\[3|V| - 7 \leq 2|V| - 3 + \frac{3|V|-7}{6},\] which after rearranging gives $|V| \leq 8-7/3$, and hence $|V| \leq 5$. It is easy to check by hand that among the non-globally $2$-rigid graphs on four or five vertices, only $K_4-e$ and $K_4^+$ satisfy $|E| = 3|V| - 7$, and that these graphs are indeed $2$-stress-independent.
\end{proof}

We need a final technical lemma before proving \cref{theorem:2dsharpbound}.

\begin{lemma}\label{lemma:separatingpair}
      Let $G = (V,E)$ be an $\mathcal{R}_2$-connected $2$-stress-independent graph. Then
      \begin{enumerate}
            \item every separating pair of vertices is nonadjacent in $G$, and
            \item for any $2$-separation $(G_1,G_2)$ of $G$, $G_1$ is not isomorphic to either $K_4^+$ or $W_5$.
      \end{enumerate}
\end{lemma}
\begin{proof}
      \emph{(a)} Suppose, for a contradiction, that $u,v$ is a separating pair of vertices of $G$ and $uv \in E$. We can deduce from \cref{lemma:Mconnected2sum} that $G-uv$ is also $\mathcal{R}_2$-connected and satisfies $\kappa(u,v;G-uv) \geq 3$. Hence by \cref{theorem:2dstresslinked}, $\{u,v\}$ is $2$-stress-linked in $G-uv$, contradicting the assumption that $G$ is $2$-stress-independent.
      
      \emph{(b)} Let $V(G_1) \cap V(G_2) = \{u,v\}$. By part \emph{(a)}, $u$ and $v$ are nonadjacent in $G$. Let us first suppose for a contradiction that $G_1$ is isomorphic to $K_4^+$. Without loss of generality we may assume that $u$ is the vertex of degree two in $G_1$. Let $w$ be a common neighbour of $u$ and $v$ in $G_1$. It is easy to verify that $G_1-vw+uv$ is isomorphic to $W_5$ and hence is $\mathcal{R}_2$-connected. It follows from \cref{lemma:Mconnected2sum} that $G-vw$ is also $\mathcal{R}_2$-connected, since it can be written as the 2-sum of $G_2+uv$ and $G_1-vw+uv$. Since $\kappa(v,w;G-vw) \geq 3$, we deduce from \cref{theorem:2dstresslinked} that $\{v,w\}$ is $2$-stress-linked in $G-vw$, contradicting the fact that $G$ is $2$-stress-independent.
      
      The case when $G_1$ is isomorphic to $W_5$ is similar. Since $u$ and $v$ are nonadjacent in $G_1$, they necessarily lie on the ``rim'' of $W_5$. Let $w$ be the vertex of degree four in $G_1$. It is again easy to verify that $G_1-vw+uv$ is isomorphic to $W_5$, and we can reach a contradiction by repeating the argument in the previous paragraph.
\end{proof}

\paragraph{Proof of \texorpdfstring{\cref{theorem:2dsharpbound}.}{Theorem 4.6}} \hspace{-1em}
      We consider two cases.
      
      \vspace{.5em}
      \noindent \textbf{\sffamily\boldmath Case 1: $G$ is not $\mathcal{R}_2$-connected.}  In this case we can use the proof idea of~\cite[Theorem 1.4]{Jear}. By adding $\mathcal{R}_2$-bridges to $G$, we may assume that it is $2$-rigid. (Recall that by \cref{lemma:stressindependentoperations}, adding $\mathcal{R}_2$-bridges preserves the property of being $2$-stress-linked.) Let $E_0$ denote the set of bridges of $\mathcal{R}_2(G)$, and let $G_i = (V_i,E_i), i \in \{1,\ldots,q\}$ be the nontrivial $\mathcal{R}_2$-components of $G$. Note that each $G_i$ has at least four vertices and is $2$-stress-independent. Now we have
      \begin{equation}\label{eq:ranksum}
            |E_0| + \sum_{i=1}^q r_2(G_i) = r_2(G) = 2|V| - 3,
      \end{equation}
      and, by using \cref{theorem:mainstrongerstress}(b) and the fact that $\mathcal{R}_2$-connected graphs are $2$-rigid,
      \begin{align}\label{eq:components}
            |E| = \sum_{i=0}^q |E_i| &\leq |E_0| + \sum_{i=1}^q \left(3|V_i| - 6\right) \nonumber \\ &= |E_0| + \sum_{i=1}^q \left(\frac{3}{2}r_2(G_i) - \frac{3}{2} \right) \nonumber \\ &= \frac{3}{2}\left(|E_0| + \sum_{i=1}^q r_2(G_i) \right) - \frac{|E_0| + 3q}{2} \nonumber \\
            &= 3|V| - \frac{9 + |E_0| + 3q}{2}  
      \end{align}
      If $q \geq 3$, then $9 + 3q \geq 18$ and we are done. If $q=0$ then $|E|\leq 2|V|-3<3|V|-9$ since $|V|\geq 8$. Hence we may assume that $1\leq q \leq 2$.
      
      We first consider the subcase when $q=1$. Let $G_1 = (V_1,E_1)$. Since $G$ is not $\mathcal{R}_2$-connected, we must have $V_1 \subsetneq V$, and by \cref{eq:ranksum}, we have $|E_0| = 2(|V| - |V_1|)$. 
      Now if $|V_1| \leq 5$, then \cref{theorem:mainstrongerstress}(b) gives
\[|E| = |E_0| + |E_1| \leq 2(|V| - |V_1|) + 3|V_1| - 6 = 3|V| - 6 - (|V| - |V_1|) \leq 3|V| - 9\] 
      since $|V|\geq 8$. Similarly, if $|V_1| \geq 6$, then \cref{theorem:mainstrongerstress}(b) and \cref{lemma:3n-7case} gives $|E_1| \leq 3|V_1| - 8$ and thus  \[|E| = |E_0| + |E_1| \leq 2(|V| - |V_1|) + 3|V_1| - 8 = 3|V| - 8 - (|V| - |V_1|) \leq 3|V| - 9.\]
      
      It remains to consider the subcase when  $q = 2$.  Since $\cR_2$-components  can have at most one vertex in common, $G_1 \cup G_2$ is not $2$-rigid. Thus $|E_0| \geq 1$, and the expression on the right in \cref{eq:components} is at most $3|V| - 8$. Moreover, from our assumption that $|V| \geq 8$ we can deduce that one of the following holds: 
      \begin{itemize}
            \item at least one of $G_1$ or $G_2$ is not a copy of $K_4$, or
            \item $G_1 \cong G_2 \cong K_4$ and $G_1 \cap G_2$ is empty, or
            \item $G_1 \cong G_2 \cong K_4$ and $G_1 \cup G_2$ does not span all of $V$.
      \end{itemize}
      If the  first alternative holds, then the inequality in \cref{eq:components} is strict by \cref{theorem:mainstrongerstress}(b), and thus $|E| \leq 3|V| - 9$. If the second or third holds then  $|E_0| \geq 3$, and we can once again deduce from \cref{eq:components} that $|E| \leq 3|V| - 9$.
      
      \vspace{.5em}
      \noindent \textbf{\sffamily\boldmath Case 2: $G$ is $\mathcal{R}_2$-connected.} Let us first consider the subcase when $G$ is not $3$-connected. Let $\{u,v\}$ be a separating vertex pair of $G$, and let $(G_1,G_2)$ be a $2$-separation of $G$ with $V(G_1) \cap V(G_2) = \{u,v\}$. It follows from \cref{lemma:stresslinkedsubgraph,lemma:separatingpair}(a), respectively, that $G_i$ is 2-stress-independent and $u$ and $v$ are nonadjacent in $G$. By \cref{lemma:Mconnected2sum}, $G_i + uv$ is $\mathcal{R}_2$-connected, and hence $|V(G_i)| \geq 4$, for $i \in \{1,2\}$. 
      
      Let $n_i = |V(G_i)|$ and $m_i = |E(G_i)|$ for $i \in \{1,2\}$. Since $G_i$ is a non-complete $2$-stress-independent graph on at least four vertices, \cref{theorem:mainstrongerstress}(b) implies that $m_i \leq 3n_i - 7$.  Moreover, we have $|V| = n_1 + n_2 - 2$, and hence
\[|E| = m_1+m_2 \leq 3(n_1+n_2) - 14 = 3|V| - 8.\] If $m_i < 3n_i - 7$ holds for some $i \in \{1,2\}$, then the inequality above is strict, and we are done.
      
      Let us assume, to the contrary, that $m_1 = 3n_1-7$ and $m_2 = 3n_2-7$. It follows from \cref{lemma:3n-7case} that $G_1$ and $G_2$ are both isomorphic to one of $K_4-e,K_4^+$ or $W_5$. Moreover, since $|V| \geq 7$, at least one of them, say $G_1$, is not isomorphic to $K_4-e$. \cref{lemma:separatingpair}(b) now implies that $G_1$ cannot be isomorphic to $K_4^+$ or $W_5$ either, a contradiction.
      
      Finally, let us consider the case when $G$ is $\mathcal{R}_2$-connected and $3$-connected. If $G$ is minimally $\mathcal{R}_2$-connected, then we are done by \cref{theorem:R2connectedsharp}. Hence we may assume that there exists an edge $e \in E$ such that $G-e$ is $\mathcal{R}_2$-connected. Since $G$ is $2$-stress-independent, \cref{theorem:2dstresslinked} implies that there exists a vertex pair $\{u,v\}$ in $G-e$ that separates the endvertices of $e$. By \cref{lemma:separatingpair}(a), $u$ and $v$ are nonadjacent in $G-e$.
      
      Let $(G_1,G_2)$ be a $2$-separation of $G-e$ with $V(G_1) \cap V(G_2) = \{u,v\}$. Let $n_i = |V(G_i)|$ and $m_i = |E(G_i)|$ for $i \in \{1,2\}$. As in the previous case, $G_i$ is a non-complete $2$-stress-independent graph on at least four vertices, and hence \cref{theorem:mainstrongerstress}(b) implies that $m_i \leq 3n_i - 7$. Moreover, we have $|V| = n_1 + n_2 - 2$ and
\[|E| = m_1+m_2 + 1 \leq 3(n_1+n_2) - 14 = 3|V| - 7.\] Thus we are done if $m_i < 3n_i - 7$ holds for both $i \in \{1,2\}$.
      
      Suppose, for a contradiction, that $m_i = 3n_i-7$ for some $i \in \{1,2\}$, say $i=1$. By \cref{lemma:separatingpair}(b), $G_1$ cannot be isomorphic to either $K_4^+$ or $W_5$, so by \cref{lemma:3n-7case}, $G_1$ must be isomorphic to $K_4 - uv$. Let $w$ be the endvertex of $e$ lying in $G_1$. It is easy to see that $G-uw$ can be obtained from $G_2+uv$ by a pair of $2$-dimensional edge split operations. By \cref{lemma:Mconnected2sum}, $G_2+uv$ is $\mathcal{R}_2$-connected, and hence by \cref{lemma:edgesplit}, so is $G-uw$. The $\mathcal{R}_2$-connectivity of $G_2+uv$ also implies that $G_2$ is $2$-rigid and hence $2$-connected, from which we can deduce that $\kappa(u,w;G-uw) \geq 3$. \cref{theorem:2dstresslinked} now implies that $\{u,w\}$ is $2$-stress-linked in $G-uw$, contradicting the hypothesis that $G$ is $2$-stress-independent. This final contradiction completes the proof.
\qed

\section{Minimally \texorpdfstring{$\cR_d$}{Rd}-connected graphs}\label{section:Mconnected}

As we noted in the previous section, \cref{theorem:2dstresslinked} implies that every minimally $\mathcal{R}_2$-connected graph is $2$-stress-independent. We conjecture that this extends to higher dimensions.

\begin{conjecture}\label{con:new} 
      Every minimally $\mathcal{R}_d$-connected graph is $d$-stress-independent.
\end{conjecture}

In support of this conjecture, we show that parts \emph{(a)} and \emph{(b)} of \cref{theorem:mainstrongerstress} also hold for a family of graphs which includes  all minimally $\cR_d$-connected graphs. This verifies the first half of~\cite[Conjecture 4.2]{Jear}. Let us note that the second half of \cite[Conjecture 4.2]{Jear} is the analogue of \cref{conjecture:stressindependentsharp} for minimally $\mathcal{R}_d$-connected graphs.

\begin{theorem}\label{theorem:mainstronger}
      Let $G$ be a graph in which every nontrivial $\mathcal{R}_d$-component is minimally $\mathcal{R}_d$-connected. Then
      \begin{enumerate}
            \item $G$ is $\mathcal{R}_{d+1}$-independent;
            \item For all subgraphs $G'=(V',E')$  of $G$ with  $|V'|\geq d+2$, we have \[|E'|\leq r_d(G') + |V'| - d - 1 \leq (d+1)|V'|-\binom{d+2}{2},\] 
            where the first inequality holds with equality if and only if $G'=K_{d+2}$. 
      \end{enumerate}
\end{theorem}

Our proof of \cref{theorem:mainstronger} is based on the concept of ear-decompositions from matroid theory, which we introduce next.
Let ${\cal M}$ be a matroid on ground set $E$, and let $C_1, C_2, \ldots, C_t$ be a sequence of circuits of ${\cal M}$. 
Let $D_0 = \varnothing$, and $D_j = \bigcup_{i = 1}^j C_i$ for $1 \le j \le t$. 
The set $C_i - D_{i - 1}$ is called the \emph{lobe} of the circuit $C_i$, and is denoted by $\widetilde{C}_i$, for each $1 \le i \le t$. 
Following~\cite{CH}, we say that $C_1, C_2, \ldots, C_t$ is a \emph{partial ear-decomposition} of ${\cal M}$ if the following properties hold for all $2\leq i\leq t$:
\begin{enumerate}[leftmargin=3em]
      \item [\textit{(E1)}] $C_i \cap D_{i - 1} \ne \varnothing$,
      \item [\textit{(E2)}]  $C_i - D_{i - 1} \ne \varnothing$, 
      \item [\textit{(E3)}] no circuit $C_i'$ satisfying \textit{(E1)} and \textit{(E2)} has $C_i' - D_{i - 1}$ properly contained in $C_i - D_{i - 1}$.
\end{enumerate}

\noindent
An \emph{ear-decomposition} of ${\cal M}$ is a partial ear-decomposition with
$D_t = E$. 
Parts \textit{(a)} and \textit{(b)} of the following lemma are from~\cite{CH}.
Part \textit{(c)} follows easily from the definitions.

\begin{lemma}\label{lemma:eardeco} 
      Let $\cal{M}$ be a matroid with
      rank function $r$. Then
      \begin{enumerate}
            \item $\cal M$ is connected if and only if $\cal M$  has an ear-decomposition.
            \item If $\cal M$ is connected, then any partial ear-decomposition of $\cal M$
            can be extended to an ear-decomposition of $\cal M$.
            \item If $C_1,C_2,\ldots,C_t$ is an ear-decomposition of $\cal M$, then
            \begin{equation*}
                  r(D_i)-r(D_{i-1})=|\widetilde C_i|-1\ \hbox{ for }\ \  1\leq i\leq t.
            \end{equation*}
      \end{enumerate}
\end{lemma}

The next lemma is from~\cite{Jear}. 

\begin{lemma}\label{minimalpartial}
      Let ${\cal M}$ be a minimally connected matroid on ground set $E$ with $|E|\geq 2$ and let
$C_1,C_2,\ldots,C_t$ be an ear-decomposition of 
${\cal M}$.
      Then
      \begin{enumerate}
            \item $|\widetilde C_i|\geq 2$ for all $1\leq i\leq t$,
            \item ${\cal M}|_{D_i}$ is minimally connected for all $1\leq i\leq t$.
      \end{enumerate}
\end{lemma}

\paragraph{Proof of \texorpdfstring{\cref{theorem:mainstronger}.}{Theorem 5.2}} \hspace{-1em}
      \textit{(a)} It follows from \cref{theorem:Mconnecteddimensiondropping} that every $\mathcal{R}_d$-component of $G$ is the union of some $\mathcal{R}_{d+1}$-components of $G$. Hence it suffices to show that each $\mathcal{R}_d$-component of $G$ is $\mathcal{R}_{d+1}$-independent. In the case of a trivial $\mathcal{R}_d$-component this is clear. Thus we may restrict our attention to a nontrivial $\mathcal{R}_d$-component, and hence assume that $G$ is minimally $\mathcal{R}_d$-connected. 
      
      By \cref{lemma:eardeco}, $\cR_d(G)$ has an ear-decomposition $C_1,C_2,\ldots, C_t$. We proceed by induction on $t$. If $t=1$, then $G$ is an $\cR_d$-circuit and every pair of edges of $G$ is a cocircuit of $\cR_d(G)$. \cref{lem:bridges}(b) now implies that every edge of $G$ is a bridge in $\mathcal{R}_{d+1}(G)$, and hence $G$ is $\mathcal{R}_{d+1}$-independent. Thus we may assume that $t\geq 2$.
      
      Put $G'=G[\bigcup_{i=1}^{t-1} C_i]$ and $J=E\sm\bigcup_{i=1}^{t-1} C_i$. Then $G'$ is $\mathcal{R}_{d+1}$-independent by \cref{minimalpartial}(b) and induction, and hence $r_{d+1}(E\sm J)=|E\sm J|$. In addition, $|J|\geq 2$ by \cref{minimalpartial}(a). The fact that $C_1,C_2,\ldots,C_t$ is an ear-decomposition of $\cR_d(G)$ and \cref{lemma:eardeco}(c) imply that every pair of edges of $J$ is a cocircuit of $\cR_d(G)$. Thus every edge of $J$ is a bridge of $\mathcal{R}_{d+1}(G)$ by \cref{lem:bridges}(b). This implies that $r_{d+1}(E)=|E\sm J|+|J|=|E|$, and hence $G$ is $\mathcal{R}_{d+1}$-independent.
      
      \textit{(b)} We proceed as in the proof of \cref{theorem:mainstrongerstress}(b). Let us add a set $F$ of $\mathcal{R}_d$-bridges to $G'$ so that $G'+F$ is $d$-rigid. Every nontrivial $\mathcal{R}_d$-component of $G'+F$ is minimally $\mathcal{R}_d$-connected, and hence part \textit{(a)} implies that $G'+F$ is $\mathcal{R}_{d+1}$-independent. Hence by a simple calculation we have,
      \begin{align}\label{eq:bound}
            |E'| = |E(G'+F)| - |F| &\leq (d+1)|V'|-\binom{d+2}{2} - |F| \\ \notag &= r_d(G'+F) + |V'| - d - 1 - |F| \\ \notag &= r_d(G') + |V'| - d - 1, 
      \end{align}
      which gives the first part of \emph{(b)}.
      
      If equality holds in \cref{eq:bound}, then $G'+F$ is minimally $(d+1)$-rigid, and hence globally $d$-rigid by \cref{theorem:rigidgloballyrigid}. \cref{theorem:mconnected} now implies that $F = \varnothing$. If $G'$ is minimally globally $d$-rigid, then $G'=K_{d+2}$ by \cref{theorem:minimallygloballyrigid}. Hence we may assume that  $G'-e$ is globally $d$-rigid for some $e\in E'$. We can now use \cref{theorem:mconnected} to deduce that both $G'$ and $G'-e$ are $\Rd$-connected.
      
      It follows that $G'$ is contained in some $\mathcal{R}_d$-component $G_0$ of $G$. Note that $G_0$ is necessarily nontrivial, and hence minimally $\mathcal{R}_d$-connected. By using parts \emph{(a)} and \emph{(b)} of \cref{lemma:eardeco},   we can construct an ear-decomposition ${\cal C}'$ of $G'$ and then extend it to an ear-decomposition of $G_0$. It follows from \cref{minimalpartial}(b) that $G'=\bigcup_{C\in {\cal C}'}C$ is minimally $\cR_d$-connected, contradicting the fact that $G'-e$ is also $\mathcal{R}_d$-connected.
      \qed

\medskip
Note
that, besides minimally $\mathcal{R}_d$-connected graphs, minimally redundantly $d$-rigid graphs and, more generally, minimally $\mathcal{R}_d$-bridgeless graphs also have the property that each of their nontrivial $\mathcal{R}_d$-components  is minimally $\mathcal{R}_d$-connected. Hence \cref{theorem:mainstronger} can  be applied to these graphs as well. On the other hand, the fact that these graphs are $\mathcal{R}_{d+1}$-independent follows immediately from \cref{lem:bridges}(b). Let us also note that \cite[Theorem 8.3]{CJT} gives a slightly stronger upper bound on the number of edges in a minimally $\mathcal{R}_3$-connected graph than the $d=3$ case of \cref{theorem:mainstronger}(b).

To close this section, we point out that an affirmative answer to the following conjecture from~\cite{stresslinked} would imply \cref{con:new}.

\begin{conjecture}\label{conjecture:stresslinkedMcomponent}\cite[Conjecture 6.4]{stresslinked}
      Let $G$ be a graph, and let $u,v$ be a pair of vertices in $G$. If $\{u,v\}$ is $d$-stress-linked in $G$, then there is some $\mathcal{R}_d$-component $G_0$ of $G$ such that $\{u,v\}$ is $d$-stress-linked in $G_0$.
\end{conjecture}

\begin{proposition}\label{prop:minimallyMconnectedstresslinked}
      Suppose \cref{conjecture:stresslinkedMcomponent} is true. Let $G$ be a graph, let $G_0$ be a nontrivial $\mathcal{R}_d$-component of $G$, and let $uv$ be an edge of $G_0$ such that $G_0-uv$ is not $\mathcal{R}_d$-connected. Then $\{u,v\}$ is not $d$-stress-linked in $G-uv$. In particular, if every nontrivial $\mathcal{R}_d$-component of $G$ is minimally $\mathcal{R}_d$-connected, then $G$ is $d$-stress-independent. 
\end{proposition}
\begin{proof}
      Suppose, for a contradiction, that $\{u,v\}$ is $d$-stress-linked in $G-uv$. By \cref{conjecture:stresslinkedMcomponent}, there is an $\mathcal{R}_d$-component $H_0$ of $G-uv$ such that $\{u,v\}$ is $d$-stress-linked in $H_0$. In particular, $\{u,v\}$ is $d$-linked in $H_0$, and hence there exists an $\mathcal{R}_d$-circuit $C$ in $H_0+uv$ containing $uv$. 
      It follows that $H_0 + uv$ is $\mathcal{R}_d$-connected, and thus  is contained in $G_0$. 
      
      Suppose $G_0\neq H_0+uv$. Let $e$ be an edge of $G_0$ not in $H_0+uv$. Since $G_0$ is $\mathcal{R}_d$-connected, there exists an $\mathcal{R}_d$-circuit $C'$ in $G_0$  with $\{uv,e\} \subseteq  E(C')$. Using the strong circuit exchange axiom on $(C \cup C') - uv$ and $e$, we obtain an $\mathcal{R}_d$-circuit $C_0$ in $G_0-uv$ that contains $e$, as well as at least one edge from $H_0$. Then $H_0\cup C_0$ is $\mathcal{R}_d$-connected. This contradicts the fact that $H_0$ is an $\mathcal{R}_d$-component of $G-uv$.
      
      Hence $G_0= H_0+uv$. Since $H_0$ is $\mathcal{R}_d$-connected, we deduce that $G_0 - uv$ is also $\mathcal{R}_d$-connected, a contradiction. 
\end{proof}

\section{Closing remarks}\label{section:closing}

\subsection{A bonus result}

The following proposition is an immediate consequence of \cref{theorem:vertexredundantlylinkedstronger}.

\begin{proposition}\label{proposition:contrapositive}
      Let $G$ be a $d$-stress-independent graph, $uv \in E(G)$, 
      and put \[U = \bigcap\{V(C): C \text{ is an } \mathcal{R}_d\text{-circuit in } G \text{ with } uv \in E(C)\} 
      .\] 
      If $uv$ is not an $\mathcal{R}_d$-bridge in $G$, then $|U| \geq d+2$.
\end{proposition}

We can use \cref{proposition:contrapositive} to prove the following.

\begin{theorem}\label{theorem:noK4}
      If $G$ is an $\mathcal{R}_2$-connected $2$-stress-independent graph on at least five vertices, then $G$ does not contain a copy of $K_4$.  
\end{theorem}
\begin{proof}
      Suppose, for a contradiction, that $G$ contains a subgraph $K\cong K_4$. Let us fix vertices $u,v \in V(K)$. Since $G$ is $\mathcal{R}_2$-connected and $G \neq K$, there exists another $\mathcal{R}_2$-circuit $C$ in $G$ with $uv \in E(C)$. It follows from \cref{proposition:contrapositive} that $V(K) \subseteq V(C)$. Since $K$ cannot be a subgraph of $C$, there exists a pair of vertices $x,y \in V(K)$ such that $xy \notin E(C)$. %
      
      Let us consider the subgraph $H = C \cup K$ of $G$. Since $G$ is $2$-stress-independent, so is $H$, and hence $\{x,y\}$ is not $2$-stress-linked in $H-xy$. Since $H - xy$ is a supergraph of $C$, it is also $\mathcal{R}_2$-connected, and thus \cref{theorem:2dstresslinked} implies that \[2 \geq \kappa(x,y;H-xy) \geq \kappa(x,y;K-xy) = 2,\] from which we deduce that $\{u,v\}$ is a separating vertex pair of $H-xy$. But by \cref{lemma:separatingpair}(a), every separating vertex pair in a $2$-stress-independent graph must be nonadjacent, a contradiction.   
\end{proof}

Let us note that  we can give an alternative proof of \cref{theorem:noK4} in the special case when $G$ is a minimally globally $2$-rigid graph using \cref{theorem:jacksonjordan}. We omit the details.

It is unclear whether \cref{theorem:noK4} generalises to higher dimensions. We pose the special case of minimally globally $d$-rigid graphs as a conjecture.

\begin{conjecture}
      If $G$ is a minimally globally $d$-rigid graph on at least $d+3$ vertices, then $G$ does not contain a copy of $K_{d+2}$.
\end{conjecture}

\subsection{Minimally \texorpdfstring{$k$}{k}-connected graphs}\label{subsection:connectivity}

We relate \cref{theorem:mainstrongerstress,theorem:2dsharpboundstrong} to the existing literature on the sparsity of minimally $k$-connected graphs.
Let $G=(V,E)$ be a minimally $k$-connected graph. As mentioned in the introduction, Mader~\cite{mader_1972a} proved that $|E|\leq k|V| - \binom{k+1}{2}$, and that if $|V|\geq 3k-2$, then
the stronger bound $|E|\leq k(|V|-k)$ applies. The latter result is tight, as shown by the complete bipartite graph $K_{k,n-k}$ for $n \geq 2k$. 

Mader's result can be extended to subgraphs by a result of Nishizeki and Poljak~\cite{nishizeki.poljak_1994} which tells us
 that $E$ can be covered by $k$ forests and hence implies that
$|E(H)|\leq k|V(H)|-k$ for all subgraphs $H\subseteq G$.
\cref{cor:d-stress} and 
\cref{theorem:mainstrongerstress} immediately give the following stronger result.
For completeness, we present a more direct proof. 

\begin{theorem}\label{theorem:mainstronger2}
      Let $k$ be a positive integer, and let $G$ be a 
      minimally $k$-connected graph. Then
      \begin{enumerate}
            \item $G$ is $\mathcal{R}_{k}$-independent;
            \item For all subgraphs $G'=(V',E')$  of $G$ with  $|V'|\geq k+1$, we have \[|E'|\leq r_{k-1}(G') + |V'| - k \leq k|V'|-\binom{k+1}{2},\] 
            where the first inequality holds with equality if and only if $G'=K_{k+1}$. 
      \end{enumerate}
\end{theorem}
\begin{proof}
      \textit{(a)} It is known that if $uv$ is an edge in an $\mathcal{R}_k$-circuit $H$, then $\kappa(u,v;H-uv)\geq k$ holds (see, e.g.,~\cite[Lemma 2.5]{jackson.jordan_2005a}).  
      The minimality of $G$ implies that $\kappa(u,v;G-uv) \leq k-1$
      for all $uv\in E(G)$. Thus  $G$ is $\mathcal{R}_{k}$-independent.

      \textit{(b)} 
      We proceed as in the proofs of \cref{theorem:mainstrongerstress}(b) and \cref{theorem:mainstronger}(b). Let us add a set $F$ of $\mathcal{R}_{d-1}$-bridges to $G'$ so that $G'+F$ becomes $(k-1)$-rigid. The edges in $F$ are $\mathcal{R}_k$-bridges as well, and hence part \textit{(a)} implies that $G'+F$ is $\mathcal{R}_{k}$-independent. As before, this implies       
      \begin{align}\label{eq:bound2}
            \notag |E'| = |E(G'+F)| - |F| &\leq (d+1)|V'|-\binom{d+2}{2} - |F| \\ \notag &= r_d(G'+F) + |V'| - d - 1 - |F| \\ \notag &= r_d(G') + |V'| - d - 1, 
      \end{align}
      as claimed. If equality holds, then $G'+F$ is $k$-rigid, and thus globally $(k-1)$-rigid by \cref{theorem:rigidgloballyrigid}. It follows from \cref{theorem:mconnected} that $F = \varnothing$. If $G'$ is minimally globally $(k-1)$-rigid, then $G'=K_{k+1}$ by \cref{theorem:minimallygloballyrigid}. Hence we may assume that  $G'-e$ is globally $(k-1)$-rigid for some $e\in E'$. Since every globally $(k-1)$-rigid graph is $k$-connected, it follows
      that $\kappa(u,v;G-uv)\geq \kappa(u,v;G'-uv)\geq k$, where $e=uv$, which contradicts the minimality of $G$. This completes the proof.      
\end{proof}

\cref{theorem:mainstronger2} shows that Mader's bound on the size of a minimally $k$-connected graph $G$ holds for all subgraphs $H$ of $G$ of order at least $k$. It is natural to ask whether Mader's stronger bound for sufficiently large graphs can also
be extended to subgraphs. This question was recently answered affirmatively by Lou and He~\cite[Theorem 2.2]{lou.he_2025}.

\begin{theorem}\label{lh}
      Let $G$ be a minimally $k$-connected graph. Then, for all subgraphs $G' = (V',E')$ of $G$ with $|V'|=h\geq 5k-4$, we have
      $E' \leq k(h-k)$, with equality if and only if $G' = K_{k,h-k}$.
\end{theorem}

\cref{lh} verifies \cref{conjecture:stressindependentsharp} for the special case when $d = k-1$ and $G$ is a subgraph of a minimally $k$-connected graph. 
The lower bound on $h$ in \cref{lh} is not best possible. We believe that it can be replaced by $h\geq 3k-2$, for all $k\geq 1$. This is obvious 
for $k=1$, and for $k=2$ it is implied by the $d=1$ case of \cref{theorem:mainstronger} by using the fact that $2$-connectivity and ${\cal R}_1$-connectivity are the same.
The $k=3$ case can be shown by carefully following the proof of \cref{theorem:2dsharpbound} to verify that the graph obtained by gluing together two copies of $K_4$ at a vertex and adding an edge is the unique 2-stress-independent graph with 7 vertices and 13 edges,  and noting that this graph is not $3$-connected. We omit the details.

We close by stating the specialisation of \cref{theorem:2dsharpboundstrong} to minimally $3$-connected graphs.

\begin{theorem}\label{theorem:minimally3connectedsharp}
      Let $G$ be a minimally $3$-connected graph. For all subgraphs $G'=(V',E')$  of $G$ with  $|V'|\geq 8$, we have 
      $|E'|\leq r_2(G') + |V'| - 6 \leq 3(|V'| - 3).$  
\end{theorem}

Compared to the $k=3$ case of \cref{lh}, \cref{theorem:minimally3connectedsharp} gives a slightly better bound on the number of edges if $r_2(G')$ is low, or equivalently, if the generic realisations of $G'$ in $\R^2$ have a high degree of freedom.

\subsection*{Acknowledgements}

This work started at the NII Shonan Meeting "Theory and Algorithms in Graph Rigidity and Algebraic Statistics" in 2024. 
This material is also based upon work supported by the National Science Foundation under Grant No. DMS-1929284 while the first and second authors were in residence at the Institute for Computational and Experimental Research in Mathematics in Providence, RI, during the "Geometry of Materials, Packings and Rigid Frameworks" semester program.
The third author was supported by the
MTA-ELTE Momentum Matroid Optimization Research Group and the National Research, Development and Innovation Fund of Hungary, financed under the ELTE TKP 2021-NKTA-62 funding scheme.

\printbibliography

\end{document}